\documentclass[a4paper,12pt]{amsart}
\usepackage{fullpage}
\usepackage[T1]{fontenc} %pacote para aceitar a acentuaÃ§Ã£o feita pelo teclado
\usepackage[]{amsfonts} %pacote para estilos matemÃ¡ticos
\usepackage{amsmath}
\usepackage{epsfig}
\usepackage{amssymb,amsthm}
\usepackage{multicol}
\usepackage{graphicx,color}
\usepackage[latin1]{inputenc}

\newtheorem{theorem}{Theorem}[section]
\newtheorem{lemma}{Lemma}

\newtheorem{corollary}{Corollary}[section]

\usepackage[]{amsfonts} %pacote para estilos matemÃ¡ticos
\usepackage[]{amsmath}
\usepackage{amssymb}
\usepackage{multicol}

\newcommand{\R}{\mathbb{R}}
\newcommand{\N}{\mathbb{N}}
\newcommand{\F}{\mathcal{F}}
\newcommand{\Pp}{\mathbb{P}}
\newcommand{\E}{\mathbb{E}}
\newcommand{\sP}{\mathcal{P}}

\numberwithin{proposition}{section}
\numberwithin{lemma}{section}

\newcommand{\incre}[2]{\Delta #1_{#2}}

\title{Concentration in the Generalized Chinese Restaurant Process}
\author{A. Pereira, R. I. Oliveira, R. Ribeiro}

\begin{document}

\maketitle

\begin{abstract} 
The Generalized Chinese Restaurant Process (GCRP) describes a sequence of exchangeable random partitions of the numbers $\{1,\dots,n\}$. This process is related to the Ewens sampling model in Genetics and to Bayesian nonparametric methods such as topic models. In this paper, we study the GCRP in a regime where the number of parts grows like $n^\alpha$ with $\alpha>0$. We prove a non-asymptotic concentration result for the number of parts of size $k=o(n^{\alpha/(2\alpha+4)}/(\log n)^{1/(2+\alpha)})$. In particular, we show that these random variables concentrate around $c_{k}\,V_*\,n^\alpha$ where $V_*\,n^\alpha$ is the asymptotic number of parts and $c_k\approx k^{-(1+\alpha)}$ is a positive value depending on $k$. We also obtain finite-$n$ bounds for the total number of parts. Our theorems complement asymptotic statements by Pitman and more recent results on large and moderate deviations by Favaro, Feng and Gao.
\end{abstract}

\section{Introduction}
Models of random partitions have attracted much attention in Probability and Statistics. In this paper we study a specific family of models of random partitions called {\em generalized Chinese Restaurant processes (GCRP)}. These models were introduced by Pitman \cite{pitman1995}, \cite{pitman2002} as two-parameter generaliation of Ewens' sampling formula \cite{ewens1972}. They are also important building blocks in topic models \cite{blei2003} and other Bayesian nonparametric methods \cite{craneubiq}.

The GRCP generates a sequence of random partitions $\sP_n$ of $[n]:=\{1,\dots,n\}$ for $n=1,2,3,\dots$. We focus on a specific setting for the model where the number of parts in $\sP_n$ grows like $n^{\alpha}$ for a parameter $\alpha\in (0,1)$. Our main goal is to prove concentration for the total number of the number of parts with size $k$ in each $\sP_n$, that is:
\[N_n(k):= |\{A\in \sP_n\,:\, |A|=k\}|.\]
As we explain below, the $\sP_n$ are {\em mixtures of i.i.d. models}, and the above random variables do not concentrate around any fixed value. Nevertheless, we show that they do concentrate around random values. Our main result -- Theorem \ref{thm:main} below -- shows that, for large $n$, with high probability,  
\[N_n(k) = c\,V_*\,\frac{\Gamma(k-\alpha)}{\Gamma(k+1)}\,n^\alpha + o\left(\frac{\Gamma(k-\alpha)}{\Gamma(k+1)}\,n^\alpha\right)\]
where $V_*$ is a random variable with $V_*>0$ a.s. and and $c>0$ is a constant depending on model parameters. This result holds simultaneously for all $k$ in a range that grows polynomially in $k$. Since \[\Gamma(k-\alpha)/\Gamma(k+1)=\Theta(k^{-(1+\alpha)})\mbox{ for large $k$},\] we verify that the power-law-type behavior in $k$ that is known to hold asymptotically for the $N_n(k)$ is already visible for finite $n$. Moreover, in our proof we also obtain finite-$n$ bounds on the number of parts in $\sP_n$ (cf. Theorem \ref{thm:V} below). 

Our proof method is based on martingale inequalities and is inspired by the analysis of preferential-attachment-type models \cite{chunglucomplex}. However, there are some important technical differences, which we discuss in subsection \ref{sec:proofoutline}. A salient feature of our approach is that the concentration-of-measure arguments we employ are somewhat delicate, and rely on Freedman's concentration inequality \cite{freedman1975}. 

The remainder of the paper is organized as follows. We fix some notation in the next paragraph. In section \ref{sec:themodel}, we introduce the model, discuss its regimes, and give some background on its theory and applications. Section \ref{sec:results} states our main theorems. We will also outline their proofs and compare them with previous results. Section \ref{prelim} contains the main concentration-of-meausre results we will need, including Freedman's inequality. Actual proofs start in Section \ref{numberparts} with the analysis of the number of parts in $\sP_n$. The arguments for $N_n(k)$ is more convoluted and takes four sections. Section \ref{numbersizek} gives some preliminary results, including a recursive formula. Section \ref{sec:UperXnk} obtains high-probability upper bounds and lower bounds for $N_n(k)$. The proof of our main Theorem is wrapped up in Section \ref{proofmain}. The final section contains some concluding remarks. The appendix collects several technical estimates \\

\noindent {\bf Notation:} In this paper $\N=\{1,2,3\dots\}$ is the set of positive integers. Given $n\in\N$, we let $[n]:=\{1,\dots,n\}$ denote the set of all numbers from $1$ to $n$. Given a nonempty set $S$, a {\em partition} $\sP$ of $S$ is a collection of pairwise disjoint and nonempty subsets of $S$ whose union is all of $S$. The elements of $\sP$ are called the {\em parts}. We denote the cardinality of a finite set $S$ by $|S|$. In particular, for a finite partition $\sP$, $|\sP|$ denotes the number of parts in $S$. Finally, when we talk about sequences $\{x_n\}_{n=0}^{+\infty}$ of random or deterministic values, we will write $\Delta x_n:=x_n-x_{n-1}$. 

\section{The model} \label{sec:themodel}

\subsection{Definitions}
Fix two parameters $\theta,\alpha\in\R$; extra conditions will be imposed later. GCRP($\alpha,\theta$)  -- shorthand for the Generalized Chinese Restaurant Process with parameters ($\alpha,\theta$) -- is a Markov chain
\[\sP_1,\sP_2,\sP_3,\sP_4,\dots\]
where, for each $n\in\N$, $\sP_n$ is a partition of $[n]:=\{1,\dots,n\}$. We let
\begin{equation}\label{eq:defVn}V_n:=|\sP_n|\end{equation} denote the number of parts in $\sP_n$ and write
\begin{equation}\label{eq:defPn}\sP_n = \{A_{i,n}\,:\,i=1,\dots,V_n\},\end{equation}
where the $A_{i,n}$ are the parts of $\sP_n$. In the colorful metaphor of the ``Chinese restaurant", the $A_{i,n}$ are the tables occupied by customers $1,\dots,n$, who arrive sequentially, with $V_n$ being the number of occupied tables. So $\sP_n$ describes the table arrangements of the first $n$ customers. 

The evolution of the process is as follows. 

\begin{itemize}
\item {\bf Initial state:} customer $1$ sits by herself i.e. $\sP_1=\{\{1\}\}$. 
\item {\bf Evolution:} Given $\sP_1,\dots,\sP_{n}$, with $\sP_n$ as in (\ref{eq:defPn}), we define $\sP_{n+1}$ via a random choice:
\begin{itemize}
\item For each $i=1,\dots,V_{n-1}$, with probability
\[\frac{|A_{i,n}| - \alpha}{n+\theta},\]
customer $n+1$ sits at the $i$th table. That is,
\[\sP_{n+1} = \{A_{j,n}\,:\,j\in [V_n]\backslash \{i\}\}\cup\{A_{i,n}\cup\{n+1\}\}.\]
Notice that $V_{n+1}=V_n$ in this case. 
\item With probability 
\[\frac{\alpha\,V_n + \theta}{n+\theta},\]
customer $n+1$ sits by herself at a new table. That is, we set
\[\sP_{n} = \{A_{i,n}\,:\, i=1,\dots,V_{n}\}\cup\{\{n+1\}\}.\]
In this case $V_{n+1}=V_n+1$.
\end{itemize}
\end{itemize}

Our focus in this paper is on $V_n$ and the random variables 
\begin{equation}\label{eq:defNnk}N_n(k):=|\{A\in \sP_n\,:\,|A|=k\}| = |\{i\in [V_n]\,:\,|A_{i,n}|=k\}|\,\,\,(k\in [n])\end{equation}
that count how many of the parts in $\sP_n$ have size $k$.

\subsection{Choices of parameters and different regimes}\label{sec:regimes}
The attentive reader will have noticed that the above process only makes sense for certain values of $\theta$ and $\alpha$. Specifically, there are different assumptions one can make, which lead to different behavior \cite{pitman1995,pitman2002}. 
\begin{itemize}
\item {\em Bounded number of parts}: if $\alpha<0$ and $\theta=-m\alpha$ for some $m\in\N$, then $V_n\to m$ almost surely. After $V_n$ reaches value $m$, the process behaves like an urn model with $m$ urns.
\item {\em Logarithmically growing number of parts}: if $\theta>0$, $\alpha=0$, then
\[\frac{V_n}{\log n}\to \theta\mbox{ almost surely}\]
and $V_n$ has Gaussian fluctuations at the scale of $\sqrt{\log n}$. 
\item {\em Polynomially growing number of parts}: 
if $\alpha>0$ and $\theta>-\alpha$, 
\begin{equation}\label{eq:defVo}\frac{V_n}{n^\alpha}\to V^o\mbox{ almost surely}\end{equation}
where $V^o$ is a nondegenerate random variable with a density over $(0,+\infty)$. In particular, $0<V^o<+\infty$ almost surely. 
\end{itemize}

This last regime is the focus of the present paper.

\subsection{Some background} 

We discuss here a bit of the history and applications of the GCRP. Those interested only in results may skip to the next section. 

The GCRP is an exchangeable model in the sense that the law of $\sP_n$ is invariant under permutations of $[n]$. One consequence of this is that the natural infinite limit $\sP_\infty$ of $\sP_n$ is an exchangeable random partition of the natural numbers $\N$. That is, the law of $\sP_\infty$ is invariant under any finite permutation of $\N$.  

A well-known result of Kingman \cite{kingman1978} says that exchangeable random partitions of $\N$ can always be built from mixtures of paintbox partitions. Suppose $P$ is a random probability distribution over $\N\cup\{\star\}$ where $\star\not\in\N$. Conditionally on $P$, let $\{X_i\}_{i\in\N}$ be an i.i.d.-$P$ sequence. Form a partition of $\N$ by placing each $i\in\N$ with $X_i=\star$ in a singleton, and (for each $k\in\N$) putting all $j$ with $X_j=k$ in the same part. Clearly, such a construction always leads to an exchangeable random partition, and Kingman's theorem says that this is the {\em only way} to build such partitions. In the specific case of the infinite GCRP($\alpha,\theta$), the law of $P$ is the two-parameter Poisson-Dirichlet distribution {\sc PD}($\alpha,\theta$). This can be used to derive explicit formulae for the distribution of $\sP_n$ for each $n$.

The GCRP was first mentioned in print by Aldous \cite{aldous1983}. It was studied by Pitman \cite{pitman1995}, \cite{pitman2002} as an example of a partially exchangeable model where many explicit calculations are possible. In particular, the exact distribution of the random variables $N_n(k)$ we consider can be computed explicitly. Based on these formulae, \cite{favaro2015}, \cite{favaro2018} obtained large and moderate deviation results for these variables. These results are briefly described in subsection \ref{sec:related} below.  

The class of models we consider is also important in many applications. On the one hand, it is a generalization of Ewens' neutral allele sampling model in population Genetics \cite{ewens1972}. On the other hand, the GCRP and its variants are important building blocks for topic models \cite{blei2003} and many other Bayesian nonparametric methods. We refer to Crane's recent survey \cite{craneubiq} for much more information on our model, its extensions and the many contexts where it has appeared. 

\section{Results}\label{sec:results}

Let $n\in\N$ and recall the definitions of $V_n$ and $N_n(k)$ in (\ref{eq:defVn}) and (\ref{eq:defNnk}), respectively. Our theorem describes these random variables in the setting where $\alpha\in (0,1)$ and $\theta+\alpha>0$. Recall from Section \ref{sec:regimes} that in this setting the random variables $n^{-\alpha}V_n$ have a nontrivial limit $V^o>0$ (cf. (\ref{eq:defVo})) . For our purposes, it is more convenient to work with the random variables $V_n/\phi_n$, where
\[
\phi_n:= \frac{\Gamma(1+\theta)}{\Gamma(1+\theta+\alpha)}\,\frac{\Gamma(n+\alpha+\theta)}{\Gamma(n+\theta)}.
\]
Note that $\phi_n/n^\alpha$ converges to a constant $c>0$ when $n\to +\infty$. In particular, the limit
\begin{equation}\label{eq:defV*}V^*:=\lim_{n\to +\infty}\frac{V_n}{\phi_n}\mbox{ almost surely}\end{equation}
exists and is a.s. positive (it is a rescaling of $V^o$). Our first result quantifies the convergence in this statement.

\begin{theorem}[Proven in subsection \ref{sec:proofVn}]\label{thm:V}
Consider a realization $\{\sP_n\}_{n\in\N}$ of the Generalized Chinese Restaurant Process GCRP($\alpha,\theta$) with parameters $\alpha\in (0,1)$ and $\theta>-\alpha$. Then there exist constants $K=K(\alpha,\theta)>0$ and $c_*=c_*(\alpha,\theta)>0$ such that for $\delta < e^{-K}$ the following holds with probability $\geq 1-\delta$:
\[\forall m\in\N\,:\, \left|\frac{V_m}{\phi_m} - V_*\right|\leq \frac{c_*\,[\log \log (m+2) + \log\left(\frac{1}{\delta}\right) ]}{(m+\theta)^{\alpha/2}}.\]
\end{theorem}

Our second and main result gives concentration of the random variables $N_n(k)$ simultaneously for all $k=o(n^{\alpha/(2\alpha+4)}/(\log n)^{1 /(\alpha+2)})$.

\begin{theorem}[Main; proven in section \ref{proofmain}]\label{thm:main} Consider a realization $\{\sP_n\}_{n\in\N}$ of the Generalized Chinese Restaurant Process GCRP($\alpha,\theta$) with parameters $\alpha\in (0,1)$ and $\theta>-\alpha$. Then there exist constants $n_0=n_0(\alpha,\theta)$, $C=C(\alpha,\theta)$ such that the following holds. Assume $n\in\N$ with $n\geq n_0$. Take $A\geq 0$, $\varepsilon\in (0,1/2)$ and define $k_{\varepsilon,n}:=\lceil \varepsilon\,n^{\alpha/(2\alpha+4)}/(\log n)^{1 /(\alpha+2)}\rceil$. Then the following holds with probability $1-e^{-A}$:
\begin{align*}
\forall k\in [k_{\epsilon,n}]\,:\,\left|N_n(k) - c(\alpha,\theta)\,\frac{\Gamma(k-\alpha)}{\Gamma(k+1)}\,V_*\,n^\alpha\right| \leq  C \frac{ \Gamma(k-\alpha)}{ \Gamma(k+1)}  n^{\alpha} \varepsilon^{\alpha+2}  \left(1 + \frac{A }{\log n}\right)
\end{align*}
where
\[c(\alpha,\theta) := \frac{\alpha\,\Gamma(1+\theta)}{\Gamma(1-\alpha)\,\Gamma(1+\alpha+\theta)}>0.\]
\end{theorem} 
The following immediate corollary is perhaps somewhat easier to parse. 
\begin{corollary}In the setting of Theorem \ref{thm:main}, let $\varepsilon=\varepsilon_n\to 0$ with $n$. Then there exist sequences $C_n\to +\infty$, $\xi_n\to 0$ such that the the probability that for some we have
\[\Pp\left(\forall k\in [k_{\varepsilon,n}]\,:\,V_*-\xi_n<\frac{N_n(k)}{c(\alpha,\theta)\,\frac{\Gamma(k-\alpha)}{\Gamma(k+1)}\,n^\alpha} < V_*+  \xi_n\right)\geq 1 - n^{-C_n}, \] 
for large enough $n\in\N$.
\end{corollary}
\begin{proof} [Proof sketch] Apply Theorem \ref{thm:main} with $A=C_n\,\log n$, where $C_n\to +\infty$ but $\varepsilon_n^{\alpha+2}\,C_n\to 0$. Then take:
\[\xi_n = \frac{C}{c(\alpha,\theta)}\,\varepsilon_n^{\alpha+2}\,(1 + C_n).\]
\end{proof}

\subsection{Related work}\label{sec:related}

One consequence of our results is the a.s. asymptotics for $N_n(k)/V_n$:
\[\frac{N_n(k)}{V_n}\to c(\alpha,\theta)\,\frac{\Gamma(k-\alpha)}{\Gamma(k+1)}.\]
This kind of Law of Large Numbers was first obtained by Pitman \cite[Chapter 3]{pitman2002} with no explicit convergence rates. 

Much more recently, Favaro, Feng and Gao \cite{favaro2015,favaro2018} have used Pitman's explicit formulae to obtain large and moderate deviation results for the $N_n(k)$. Reference \cite{favaro2018}, which is the closest to our work, focuses on precise estimates for probabilities like
\begin{equation}\label{eq:favaromoderate}\Pp\left(\frac{N_n(k)}{n^\alpha\,\beta_n}>c\right)\mbox{ when }\beta_n\gg (\log n)^{1-\alpha}.\end{equation} The paper \cite{favaro2015} considers even larger sequences $\beta_n$. By contrast, we obtain finite-$n$ estimates for deviations at smaller scales, which (as expected) are not as precise. There is also a difference in proof methods: whereas they rely on explicit formulae, our argument is based on recursions and martingales. 

Another important conceptual difference between our work and that of Favaro et al. is that, for their purposes, the lack of concentration in $V_n/\phi_n$ is not an issue. Indeed, if one goes ``deep enough"~into the tail of the $N_n(k)$, as in (\ref{eq:favaromoderate}), the nontrivial distribution of $V_*=\lim V_n/\phi_n$ becomes irrelevant. Our theorems operate at a finer scale and complement these previous papers by giving tail bounds for $V_*$ and $\sup_n V_n/\phi_n$ matter (cf. Theorem \ref{Vm}). As a result, we find in Theorem \ref{thm:main} that the sequence $\{N_n(k)\}_{k}$ is essentially a deterministic function of $V_*$. 

\subsection{Proof outline}\label{sec:proofoutline}

The general methodology in our proof is based in the study of degree distributions in preferential attachment random graphs, as in the book by Chung and Lu \cite[Chapter 3]{chunglucomplex}. However, a new phenomenon arises. In the graph setting, the total number of vertices at time $n$ is usually linear $n$ (at least with high probability). By contrast, the analogue of the total number of vertices is $V_n$ -- the number of parts --, which is sublinear and not concentrated. 

One consequence of this point in our analysis is that the martingale arguments are much more delicate, and rely on Friedman's martingale inequality (cf. section \ref{prelim}), instead of the more usual (and less precise) Azuma-H\"{o}ffding bound. Another point is that we must first obtain results on the number of parts $V_n$, which we do in section \ref{numberparts}. 

We then consider the random variables $N_n(k)$. The general strategy is to write these variables in terms of ``recursions + martingales"~depending on $N_{n-1}(i)$ for $i=k-1,k$, and then observe how the ``martingale"~part concentrates. These first steps, which are taken in section \ref{numbersizek}, are similar to the analysis in \cite[Chapter 3]{chunglucomplex}. However, the results obtained are not directly employable to prove the main theorem. Section \ref{sec:UperXnk} then turns these arguments into actionable bounds. This leads to the proof of the main result in section \ref{proofmain}.

\section{Concentration inequalities}  \label{prelim}
We recall here Freedman's inequality and a particular corollary that will be important to our proofs. 
\begin{theorem}[Freedman's Inequality \cite{freedman1975}]
Let $(M_n, \F_n)_{n\geq 1}$ be a martingale with $M_0=0$ and $R>0$ a constant. Write
\begin{align*}
W_n:= \sum_{k=2}^{n} \mathbb{E}[(\Delta M_{j})^2 |\F_j].
\end{align*}
Suppose
$$ |\Delta M_{j}| \leq R, \ \ \text{for all } j.$$
Then, for all $\lambda >0$ we have
$$ \Pp(M_n \geq \lambda, W_n \leq \sigma^2) \leq \exp \left(\frac{-\lambda^2}{2\sigma + 2R\lambda/3} \right).$$
\end{theorem}
The lemma below is a straightforward consequence of Freedman's inequality. Since we will deal with the problem of bounding martingales under some constraints frequently, it will be convenient to have this precise statement.
\begin{lemma} \label{aux}
	Let $M_j$ is a martingale and $R>0$ a constant such that, $M_0=0$,  $|M_{j+1}-M_j| \leq R \ \ \forall j \leq n$ and $W_n$ is its quadratic variation, then for any constant $c_1>0$ we have
	\begin{align*}
	\Pp \left( |M_n| \geq R \lambda \right) \leq 2\exp\left( \frac{-\lambda}{2c_1+\frac{2}{3}} \right) + \Pp\left(W_n \geq c_1R^2 \lambda \right).
	\end{align*}
	\end{lemma}
	
\begin{proof}
	It follows of the union bound and Freedman's inequality to the martingales $M_j$ and $-M_j$:
	\begin{align*}
	\Pp \left( |M_n| \geq R \lambda \right) &\leq \Pp \left( |M_n| \geq R \lambda, W_n \leq c_1 R^2 \lambda \right) + \Pp\left(W_n \geq R^2 \lambda \right) \\
	&\leq 2\exp\left( \frac{-(R \lambda)^2}{2c_1R^2\lambda+\frac{2R}{3}(R \lambda)} \right) + \Pp\left(W_n \geq c_1R^2 \lambda \right) \\
	&\leq 2\exp\left( \frac{-\lambda}{2c_1+\frac{2}{3}} \right) + \Pp\left(W_n \geq c_1R^2 \lambda \right).
	\end{align*}
\end{proof}

\section{Estimates on the number of parts} \label{numberparts}
In this section we obtain results on the number of parts $V_n$ of $\mathcal{P}_n$. In particular, we prove Theorem \ref{thm:V} above. 

In subsection \ref{sec:recurrenceVn} we prove a recurrence relation for $V_n$. We use this in subsection  \ref{sec:concentrationVn} to derive concentration for the whole sequence. Finally subsection \ref{sec:proofVn} proves Theorem \ref{thm:main}.

The following normalizing factor will appear in our proofs:
\begin{equation}\label{def:phi}
\phi_n := \prod_{j=1}^{n-1}  \left( 1 + \frac{\alpha}{j+\theta} \right) = \frac{\Gamma(1+\theta)}{\Gamma(1+\theta+\alpha)} \frac{\Gamma(n+\alpha+\theta)}{\Gamma(n+\theta)}.
\end{equation}
Note that by Lemma \ref{Ord:phin} we have $\phi_n = \Theta(n^{\alpha})$.

\subsection{A recurrence relation}\label{sec:recurrenceVn}
The first result in this section is the following Lemma.
\begin{lemma}[Recurrence relation for $V_n$]\label{recV} For all $n,m \in \N$ the recurrence relation holds
	\begin{equation}
	\frac{V_{n}}{\phi_{n}} = \frac{V_{m}}{\phi_{m}} + (M_n - M_m) + \frac{O(1)}{(m+\theta)^{\alpha}},
	\end{equation}
	where $(M_n, \F_n)$ is a martingale satisfying $M_0=0$, 
	\begin{enumerate}
		\item $|\Delta M_{j}| \leq \frac{2\Gamma(1+\theta+\alpha)}{\Gamma(1+\theta) \cdot (1+\theta)^{\alpha}}$;
		
		\item $\displaystyle \mathbb{E}[(\Delta M_j)^2|\F_{j-1}] \leq \frac{2\Gamma(1+\theta+\alpha) \alpha}{\Gamma(1+\theta) } \cdot (j+\theta)^{-\alpha-1} \left(\frac{V_{j-1} +\frac{\theta}{\alpha}}{\phi_{j-1}} \right)$,
	\end{enumerate}
	for all $j \in \N$.
\end{lemma}
\begin{proof} Recall $\incre{V}{n} = V_n - V_{n-1}$. On the other hand, we also know that
\begin{equation}\label{eq:incrementV} 
\Pp \left( \incre{V}{n} = 1 \middle | \F_{n-1}\right) = \E \left[\incre{V}{n} \middle | \F_{n-1} \right] = \frac{\alpha V_{n-1} + \theta}{n-1+\theta}.
\end{equation}
In other words, conditioned on $\F_{n-1}$, the random variable $\incre{V}{n}$ is distributed as $\text{ Be}\left(\frac{\alpha V_{n-1} + \theta}{n-1+\theta} \right)$. In order to obtain mean zero martingale, it will be useful to centralize the random variable~$\incre{V}{n}$. Thus we may write $V_{n}$ as 
\begin{equation}
	\begin{split}
		V_{n} &= V_{n-1} + \Delta V_{n} \\
		V_n& = \left(1+\frac{\alpha}{n-1+\theta} \right)V_{n-1} + \left(\Delta V_{n} - \frac{\alpha V_{n-1} + \theta}{n-1+\theta} \right) + \frac{\theta}{n-1+\theta}.
	\end{split}
\end{equation}
Thus, dividing the above identity by $\phi_n$, we obtain
\begin{equation}
\begin{split}
\frac{V_{n}}{\phi_{n}} &= \frac{V_{n-1}}{\phi_{n-1}} + \zeta_{n} + \frac{\theta}{(n-1+\theta)\phi_n},
\end{split}
\end{equation}
where 
\begin{equation}
\zeta_n := \frac{\incre{V}{n} -\frac{\alpha V_{n-1} + \theta}{n-1+\theta} }{\phi_{n}}.
\end{equation}
Observe that
\begin{equation}
\label{meanZeroZeta} \mathbb{E}[\zeta_n|\F_{n}] =0.
\end{equation}
Iterating this argument $n-m$ steps leads to
\begin{equation}\label{recVphi}
\frac{V_{n}}{\phi_{n}} = \frac{V_{m}}{\phi_{m}} + (M_n - M_m) + (\theta_n-\theta_m),
\end{equation}
where
\begin{equation}
M_n := \sum_{j=2}^{n} \zeta_j \text{ and }\theta_n = 1+ \sum_{j=1}^{n-1}\frac{\theta}{(j+\theta)\phi_{j+1}}.
\end{equation}
Notice that identity (\ref{meanZeroZeta}) implies that $M_n$ is a zero mean martingale.

Now we estimate the order of the deterministic contribution of $\theta_n - \theta_m$ on identity (\ref{recVphi}). By Lemma \ref{Ord:phin}, the following upper bound holds 
\begin{equation}
\frac{1}{(j+\theta) \phi_{j+1}} < \frac{2\Gamma(1+\theta+\alpha)}{\Gamma(1+\theta) \cdot (j+\theta)^{1+\alpha}}.
\end{equation}
Thus, bounding the sum by the integral, we obtain
\begin{equation}\label{ineq:bound1phi}
\theta_n-\theta_m = \sum_{j=m}^{n-1}\frac{\theta}{\phi_{j+1}(j+\theta)} \leq \frac{4 \Gamma(1+\theta+\alpha) \theta}{\alpha \Gamma(1+\theta)} \frac{1}{(m+\theta)^{\alpha}}.
\end{equation}
which proves the first statement of the lemma.

In the remainder of the proof we estimate the increments of the martingale $M_n$ as well as its conditioned quadratic variation. By the definition of $M_j$ and recalling that $\Delta V_j$ is at most one and the bound on (\ref{ineq:bound1phi}) we obtain that 
\begin{equation}
|\Delta M_j| \leq \dfrac{1}{\phi_j} \le \frac{2\Gamma(1+\theta+\alpha)}{\Gamma(1+\theta) \cdot (j+\theta)^{\alpha}}
\end{equation} and also
\begin{equation}
\begin{split}
\mathbb{E}[(\Delta M_{j})^2|\F_{j-1}] &\leq \frac{\alpha}{(j-1+\theta)\phi_{j}} \frac{\phi_{j-1}}{\phi_{j}} \frac{V_{j-1} +\frac{\theta}{\alpha}}{\phi_{j-1}}\\
&\leq \frac{2\Gamma(1+\theta+\alpha) \alpha}{\Gamma(1+\theta) } \cdot (j+\theta)^{-\alpha-1} \frac{V_{j-1} +\frac{\theta}{\alpha}}{\phi_{j-1}},
\end{split}
\end{equation}
which proves the lemma.
\end{proof}

\subsection{Concentration and tail bounds}\label{sec:concentrationVn}
We combine the recurrence relation we have proven with Freedman's inequality to obtain the following theorem
\begin{theorem} \label{Vm}
	In the $(\alpha, \theta)$-GCRP there are constants $K=K(\alpha,\theta)>0$ and $c_V=c_V(\alpha,\theta)>0$ such that for all $m \geq 0$ integer and $A \geq K$ we have 
	$$ \Pp\left( \sup_{j \geq m} \left(\frac{V_j}{\phi_j} - \frac{V_m}{\phi_m} \right) \geq  \frac{A}{(m+\theta)^{\alpha/2}} \right)\leq \exp (-c_VA).$$
	In particular, for $m=0$, considering $V_0=0$ and $\phi_0=1$ we have
	$$ \Pp \left[ \sup_{j \in \N} \left( \frac{V_j}{\phi_j}\right) \geq A \right] \leq \exp (-c_VA).$$
\end{theorem}
\begin{proof}
We start with the particular case $m=0$ and then use it to prove the general result.
	
\noindent \textit{Case $m=0$.}	From Lemma \ref{recV} we know that the $V_n$ may be written as a mean zero martingale $M_n$ plus a deterministic factor $\theta_n$, where $\{\theta_n\}_{n \in \N}$ is a increasing positive and bounded sequence of real numbers. Thus, $\{\theta_n\}_{n \in \N}$ converges to some positive number~$\theta_{\infty}$. For a positive real number $A$, consider the following stopping time
\begin{align}\label{def:tauA}
T_A:&= \inf \left\{i \in \N : \frac{V_{i}}{\phi_{i}} \geq A + \theta_i\right\}.
\end{align}
Observe that
\begin{equation}\label{ineq:boundsupprob}
	\begin{split}
\Pp\left( \sup_{j \in \N} \left( \frac{V_j}{\phi_j} \right) \geq A+ \theta_{\infty} \right) &\leq \Pp\left( \exists j \in \N: \frac{V_j}{\phi_j} \geq A+ \theta_{j} \right) \\
&= \lim_n \Pp\left(\frac{V_{T_A\wedge n}}{\phi_{T_A\wedge n}} \geq A+ \theta_{T_A\wedge n} \right)\\
&= \lim_n \Pp \left(M_{T_A\wedge n}  \geq A \right).
\end{split}
\end{equation}
By the above inequality, the first case is proven if we obtain a proper upper bound for the tail of the stopped martingale~$\{M_{T_A \wedge n}\}_{n\in \N}$. We will do this via Lemma \ref{aux}, which requires bounds on the increment and quadratic variation of $\{M_{T_A \wedge n}\}_{n\in \N}$. We obtain these bounds on the next lines. For the increment a direct application of Lemma \ref{recV} gives us
\begin{align*}
|M_{T_A\wedge (j+1)} -M_{T_A\wedge j}| \leq R,
\end{align*}
where $R=  \dfrac{2 \Gamma(1+\theta+\alpha) }{\Gamma(1+\theta) } (1+\theta)^{-\alpha}$. For the quadratic variation $W_{n \wedge T_A}$ we have that, also by Lemma \ref{recV}, 
\begin{align}\label{ineq:wnta}
\nonumber
W_{n \wedge T_A} & = \sum_{j=2}^{n \wedge T_A} \mathbb{E}[(\Delta M_j)^2|\F_{j-1}] \\
\nonumber
& \leq  \sum_{j=2}^{n \wedge T_A} \frac{2\Gamma(1+\theta+\alpha) \alpha }{\Gamma(1+\theta) } \cdot (j+\theta)^{-\alpha-1} \frac{V_{j-1} +\frac{\theta}{\alpha}}{\phi_{j-1}}\\
& \leq  \sum_{j=2}^{n \wedge T_A} \frac{2\Gamma(1+\theta+\alpha) \alpha }{\Gamma(1+\theta) } \cdot (j+\theta)^{-\alpha-1} \left( A + \theta_j +  \frac{\theta}{\alpha \phi_{j-1}} \right).
\end{align}
Choosing $A \geq K(\alpha, \theta)$, which is defined below:
\begin{align}
\label{def:K} K(\alpha,\theta):=  \frac{\theta}{\alpha} + \sup_{j \in \N} \{ \theta_j \};
\end{align}
on (\ref{ineq:wnta}), we obtain
\begin{align*}
W_{n \wedge T_A} & \leq  \frac{4\Gamma(1+\theta+\alpha) }{\Gamma(1+\theta) } (1+\theta)^{-\alpha} \cdot A = \frac{2}{R} \cdot (R^2 A).
\end{align*}
Finally, applying Lemma \ref{aux}, with 
\begin{align*}
c_1 :=\frac{2}{R} =  \dfrac{\Gamma(1+\theta) } {  \Gamma(1+\theta+\alpha) }(1+\theta)^{\alpha}
\end{align*}
we obtain
\begin{equation}
\Pp \left(M_{T_A\wedge n}  \geq A \right) \leq \exp \left(\frac{-A}{ 2\dfrac{\Gamma(1+\theta) } {  \Gamma(1+\theta+\alpha) }(1+\theta)^{\alpha} +\frac{2}{3}} \right),
\end{equation}
and
\begin{equation}
\Pp \left(M_{T_A\wedge n}  \geq A \right) \leq \exp \left( -c_2 A \right),
\end{equation}
for 
\begin{align}
c_2 = \left( 2\dfrac{\Gamma(1+\theta) } {  \Gamma(1+\theta+\alpha) }(1+\theta)^{\alpha} +\frac{2}{3} \right)^{-1}.
\end{align} 
 The above inequality combined with (\ref{ineq:boundsupprob}) gives us
\begin{align*}
\Pp\left( \sup_{m \in \N} \left( \frac{V_m}{\phi_m} \right) \geq A \right) &\leq \exp(-c_2 A),
\end{align*}
proving the result for $m=0$.

\noindent \textit{Case $m>0$}.
The proof of the case $m>0$ is similar to the first case, but it requires another stopping time and the case $m=0$ itself. So, consider the following stopping time:
\begin{align*}
\hat{T}_B :&= \inf \left\{j \geq m : \frac{V_{j}}{\phi_{j}} - \frac{V_{m}}{\phi_{m}}  \geq B \right\} \\
&=  \inf \left\{j \in \N : (M_j-M_m) + (\theta_j - \theta_m)  \geq B \right\}.
\end{align*}
Observe that, as showed in the proof of Lemma \ref{recV},
\begin{equation}
\theta_n-\theta_m  \leq \frac{4 \Gamma(1+\theta+\alpha) \theta}{\alpha \Gamma(1+\theta)} \frac{1}{(m+\theta)^{\alpha}}.
\end{equation}
Now, let $\displaystyle B= \frac{A}{(m+\theta)^{\alpha/2}}$ and suppose $A \geq 2\theta_{\infty}$. Thus,
\begin{align*}
\Pp\left( \sup_{j \geq m} \left(\frac{V_j}{\phi_j} - \frac{V_m}{\phi_m} \right) \geq  \frac{A}{(m+\theta)^{\alpha/2}} \right) 
& \leq  \Pp\left( \exists j \leq n: (M_j - M_m) + (\theta_j - \theta_m)  \geq  \frac{A}{(m+\theta)^{\alpha/2}}  \right)\\
& =   \Pp\left(  (M_{\hat{T}_B \wedge n} - M_m) + (\theta_{\hat{T}_B \wedge n} - \theta_m)  \geq B  \right) \\
\mbox{(use that $\theta_{\hat{T}_B \wedge n}\geq \theta_m$)}& \leq   \lim_n \Pp \left(M_{\hat{T}_B \wedge n}- M_m \geq  \frac{A}{2(m+\theta)^{\alpha/2}} \right).
\end{align*}
Let $T_A$ be the same as defined in (\ref{def:tauA}). Then:
\begin{align*}
\Pp \left(M_{\hat{T}_B \wedge n}- M_m \geq  \frac{A}{2(m+\theta)^{\alpha/2}} \right) & \le  \Pp \left(M_{\hat{T}_B \wedge n} - M_m \geq  \frac{A}{2(m+\theta)^{\alpha/2}}, T_A \geq n \right) \\ & \hspace{0.5cm}+  \Pp( T_A < n )\\
& \leq   \Pp \left(M_{\hat{T}_B  \wedge T_A \wedge n} - M_m \geq  \frac{A}{2(m+\theta)^{\alpha/2}}, T_A \geq n \right) \\
& \hspace{0.5cm}+ \Pp \left( \sup_{j \in \N} \frac{V_j}{\phi_j} \geq A \right)\\
& \leq   \Pp \left(M_{\hat{T}_B  \wedge T_A \wedge n} - M_m \geq  \frac{A}{2(m+\theta)^{\alpha/2}}  \right) \\
& \hspace{0.5cm}+ \Pp \left( \sup_{j \in \N} \frac{V_j}{\phi_j} \geq A \right).
\end{align*}
As in the case $m=0$, by Lemma \ref{Vm}, the increment of $\{M_{j \wedge \hat{T}_b \wedge T_a}\}$ satisfies the following upper bound
\begin{align*}
|(M_{(j+1) \wedge \hat{T}_B  \wedge T_A} - M_m)- (M_{j \wedge \hat{T}_B  \wedge T_A} - M_m)| &\leq   \frac{2 \Gamma(1+\theta+\alpha) }{\Gamma(1+\theta) } (m+\theta)^{-\frac{\alpha}{2}},
\end{align*}
whereas its quadratic variation satisfies
\begin{equation*}
W_{n \wedge T_A} \leq  \frac{4\Gamma(1+\theta+\alpha) }{\Gamma(1+\theta) } (m+\theta)^{-\alpha} \cdot A.
\end{equation*}
Thus, again by Lemma \ref{aux} it follows that
\begin{align*}
 \Pp \left(M_{\hat{T}_B  \wedge T_A \wedge n} - M_m \geq  \frac{A}{2(m+\theta)^{\alpha/2}}  \right) \leq \exp(-c_3 A),
\end{align*}
for some constant $c_3$, which implies
\begin{align*}
\Pp\left( \sup_{j \geq m} \left(\frac{V_j}{\phi_j} - \frac{V_m}{\phi_m} \right) \geq  \frac{A}{(m+\theta)^{\alpha/2}} \right) &\leq \exp(-c_2 A) + \exp(-c_3 A) \leq \exp(-c_V A),
\end{align*}
for $c_V = \log 2 \cdot \min \{c_2,c_3\}$.
\end{proof}

\subsection{Proof of Theorem \ref{thm:V}}\label{sec:proofVn}

A consequence of Theorem \ref{Vm} is to give estimates of how large the deviation of $V_j/\phi_j$ from its limit $V_*$ can be uniformly in time.

\begin{proof}[Proof of theorem \ref{thm:V}]

Given $\delta$ define 
\begin{equation}
\delta_j =  \frac{\delta}{(j+1)(j+2)}.
\end{equation}
Let $E_j$ denote the following event
\begin{equation}
E_j := \left\{ \forall m \geq 2^j: \left| \frac{V_m}{\phi_m} - \frac{V_{2^j}}{\phi_{2^j}} \right| \leq \frac{ \log \frac{1}{\delta_j} }{c_V (2^j +\theta)^{\frac{\alpha}{2}}} \right\}.
\end{equation}

Assuming $ \log \frac{2}{\delta} \geq K_1$ we have by Theorem \ref{Vm} 
\begin{align*}
\Pp(E_j^c) &\leq \exp \left( - \log \frac{1}{\delta_j}   \right) \leq \frac{ \delta }{(j+1)(j+2)},
\end{align*}
which implies, by union bound,
\begin{align*}
\Pp \left( \bigcap_{j \geq 0} E_j \right) &\geq 1- \sum_{j \geq 0} \Pp(E_j^c) \\
&\geq 1- \sum_{j \geq 0} \delta_j \\
&\geq 1- \delta.
\end{align*}

Now, observe that, when $E_j$ occurs, we have for all $m \in [2^j, 2^{j+1}]$
\begin{align*}
\left| \frac{V_m}{\phi_m} - V_* \right| &\leq \left| \frac{V_m}{\phi_m} - \frac{V_{2^j}}{\phi_{2^j}} \right| + \left|  \frac{V_{2^j}}{\phi_{2^j}} - V_* \right|  \\
& \leq 2\sup_{m \geq 2^j} \left| \frac{V_m}{\phi_m} - \frac{V_{2^j}}{\phi_{2^j}} \right| \\
& \leq \frac{2\log \frac{1}{\delta_j} }{(2^j+\theta)^{\frac{\alpha}{2}}} \\
& \leq \frac{1}{c_V(2^j+\theta)^{\frac{\alpha}{2}}} \left[  4\log (j+2) + 2\log \left( \frac{1}{\delta} \right) \right],
\end{align*}
and once $m \in [2^j,2^{j+1}]$ it follows that
\begin{align*}
\left| \frac{V_m}{\phi_m} - V_* \right| &\leq \frac{32}{c_V(m+\theta)^{\frac{\alpha}{2}}}\left[  \log \log (m+2) + \log \left( \frac{1}{\delta} \right) \right],
\end{align*}
for any $j \in \{0,1,2,\cdots\}$.
To finish take $c_*= \frac{32}{c_V}$.
\end{proof}

\section{Preliminary estimates for the number of parts of size $k$} \label{numbersizek}
This section is devoted to give estimates for the number of classes with fixed number of elements at time $n$, $N_n(k)$. As in the case for $V_n$, we investigate the behaviour of $N_n(k)$ properly normalized. In this sense, we let $\psi_n(k)$ be the normalization factor for $N_n(k)$ given by the expression below
\begin{equation}\label{def:psik}
\begin{split}
	\psi_n(k) :&= \prod_{j=1}^{n-1}  \left( 1 - \frac{k-\alpha}{j+\theta} \right) =  \frac{\Gamma(k+\theta)\Gamma(n-k+\alpha+\theta)}{\Gamma(\alpha+\theta)\Gamma(n+\theta)}.
\end{split}
\end{equation}
We note that, for each $k$ fixed, $\psi_n(k) = \Theta(n^{\alpha-k})$. The proof of this result may be done similarly to that one given to $\phi_n$. We also let $X_n(k)$ be
\begin{equation}\label{def:Xnk}
X_n(k) := \frac{N_n(k)}{\psi_n(k)}.
\end{equation}
The first step in the analysis of the non-asymptotic behavior of $N_n(k)$ is to prove that $X_n(k)$ also satisfies a recurrence relation (Subsection \ref{sec:recurrenceXn}). We then present a martingale concentration argument that will be useful in analyzing the recurrence (Subsection \ref{sec:martingaleXn}). Subsequent sections will use these results to give upper and lower bounds on $N_n(k)$.

\subsection{Recurrence relation for $X_n(k)$}\label{sec:recurrenceXn}
The goal of this part is to derive a recurrence relation for $X_n(k)$. The proof is essentially the same we have given for $V_n$.
\begin{lemma} \label{lemma:recX}
For all $n,k \in \N$ the sequence $\{X_n(k)\}_{n \in \N}$ satisfies
\begin{align}
\label{rec:Xn1} X_{n}(1) &= M_{n}(1) + \sum_{j=1}^{n-1}\frac{\alpha V_j }{(j+\theta)\psi_{j+1}(1)}+ \theta_{n}; \\
\label{rec:Xnk} X_{n}(k) & = M_{n}(k)+X_k(k) + \frac{k-1-\alpha}{k-1+\theta} \sum_{j=k}^{n-1} X_j(k-1), \forall k > 1,
\end{align}
where $\{M_n(k)\}_{n \in \N}$ are zero mean martingales defined in (\ref{def:Mn1}) and (\ref{def:Mnk}) for all $k \in \N$ and 
\begin{equation}
\theta_n(1) :=  N_1(1)+ \displaystyle \sum_{j=1}^{n-1}\frac{\theta}{(j+\theta)\psi_{j+1}(1)}.
\end{equation}
\end{lemma}
\begin{proof} We treat the case $k=1$ separately since $X_n(1)$ satisfies a recurrence relation slight different from the other cases. However, the proof for both cases follow the recipe given by the proof of Lemma \ref{recV}, so we do not fill all the details here.

\noindent \textit{Case $k=1$.} Note that $\incre{N}{n}(1) \in \{-1,0,1\}$. Thus, conditioned to $\F_{n-1}$ we know its distribution, which is given by
\begin{equation}\label{eq:distx}
\begin{split}
\Pp\left(\incre{N}{n}(1) = -1 \middle |\F_{n-1} \right) &= \frac{(1- \alpha)N_{n-1}(1)}{n-1+\theta}; \\
\Pp\left(\incre{N}{n}(1) = 1 \middle |\F_{n-1} \right) &= \frac{\alpha V_{n-1} + \theta}{n-1+\theta}; \\
\Pp\left(\incre{N}{n}(1) = 0 \middle |\F_{n-1} \right)) &= 1- \Pp\left(\incre{N}{n}(1) = -1 \middle |\F_{n-1} \right)- \Pp\left(\incre{N}{n}(1) = 1 \middle |\F_{n-1} \right)
\end{split}
\end{equation}
Again, as in Lemma \ref{recV} but normalizing properly, define
\begin{align}
\zeta_n(1) :&= \frac{1}{\psi_{n}(1)}\left( \Delta N_n(1) -\frac{\alpha V_{n-1} + \theta - (1 - \alpha)N_{n-1}(1)}{n-1+\theta} \right),  \\
\label{def:Mn1} M_n(1) &:= \sum_{j=2}^{n} \zeta_j(1),
\end{align} 
and observe that the identities (\ref{eq:distx}) imply that the sequence $\{M_n(1)\}_n \in \N$ is a zero mean martingale. Thus
\begin{align*}
N_{n}(1) &= N_{n-1}(1) + \incre{N}{n}(1) \\ \Rightarrow 
N_{n}(1)&= \left(1 - \frac{1-\alpha}{n-1+\theta} \right)N_{n-1}(1) + \left(\incre{N}{n}(1) - \frac{\alpha V_{n-1} + \theta - (1 - \alpha)N_{n-1}}{n-1+\theta} \right) + \frac{\alpha V_{n-1} + \theta}{n-1+\theta}\\
\Rightarrow\frac{N_{n}(1)}{\psi_{n}(1)} &=  \frac{N_{n-1}(1)}{\psi_{n-1}(1)} + \zeta_{n}(k) + \frac{\alpha V_{n-1} + \theta}{(n-1+\theta) \psi_{n}(1)}.
\end{align*}
We recognize above the terms $X_{m}(1)=N_{m}(1)/\psi_{m}(1)$ for $m=n-1,n$. We conclude
\[X_{n}(1) = M_{n}(1) + \sum_{j=1}^{n-1}\frac{\alpha V_j }{(j+\theta)\psi_{j+1}(1)} + \theta_{n},\]
where 
\begin{equation}\label{def:tetan}
\theta_n(1) :=  N_1(1)+  \sum_{j=1}^{n-1}\frac{\theta}{(j+\theta)\psi_{j+1}(1)}.
\end{equation}
 \noindent \textit{Case $k>1$.} As before we calculate the conditional distribution of $\incre{N}{n}(k)$, which is given below.
\begin{align*}
\Pp\left(\incre{N}{n}(k) = -1 \middle |\F_{n-1} \right) &= \frac{(k - \alpha)N_{n-1}(k)}{n-1+\theta}; \\
\Pp\left(\incre{N}{n}(k) = 1 \middle |\F_{n-1} \right) &= \frac{(k-1 - \alpha)N_{n-1}(k-1)}{n-1+\theta}; \\
\Pp\left(\incre{N}{n}(k) = 0 \middle |\F_{n-1} \right) &= 1- \Pp\left(\incre{N}{n}(k) = -1 \middle |\F_{n-1} \right)- \Pp\left(\incre{N}{n}(k) = 1 \middle |\F_{n-1} \right);
\end{align*}
Again we centralize and normalize it and define our martingale from its sum:
\begin{align}
\zeta_{n}(k) &:= \frac{\incre{N}{n}(k) - \frac{(k - 1- \alpha)N_{n-1}(k-1) - (k - \alpha)N_{n-1}(k)}{n-1+\theta} }{\psi_{n}(k)}; \\ 
\label{def:Mnk} M_n(k) &:= \sum_{j=k+1}^{n} \zeta_j(k).
\end{align}
The relation below between $\psi_n(k-1)$ and $\psi_{n+1}(k)$ will be useful to our purposes:
\begin{equation}\label{eq:relationpsis}
\frac{\psi_{n}(k-1)}{\psi_{n+1}(k)} =\frac{n+\theta}{k-1+\theta}.
\end{equation}
This follows from the definition of $\psi_n(k)$ given at (\ref{def:psik}). This relation allows us to derive the desired recurrence relation as follows
\begin{eqnarray*} 
N_{n}(k)& = & \left(1 - \frac{k-\alpha}{n+\theta} \right)N_{n-1}(k) \\ & & + \left(\incre{N}{n}(k) - \frac{(k - 1- \alpha)N_{n-1}(k-1) - (k - \alpha)N_{n-1}(k)}{n-1+\theta} \right)\\
 & & + \frac{(k - 1- \alpha)N_{n-1}(k-1)}{n-1+\theta}  \\
\Rightarrow \frac{N_{n}(k)}{\psi_{n}(k)} &=&  \frac{N_{n-1}(k)}{\psi_{n-1}(k)} + \zeta_{n}(k) + \frac{(k - 1- \alpha)N_{n-1}(k-1)}{(n-1+\theta)\psi_{n}(k)}.
\end{eqnarray*}
We have above the terms $X_m(k)=N_m(k)/\psi_{m}(k)$ for $m=n,n+1$. The last term in the right-hand side is:
\[\frac{N_{n-1}(k-1)}{\psi_{n}(k)} = \frac{\psi_{n-1}(k-1)}{\psi_{n}(k)}\, X_{n-1}(k-1)= \frac{n-1+\theta}{k-1+\theta}\, X_{n-1}(k-1)\mbox{ by (\ref{eq:relationpsis}).}\]
We deduce:
\[X_{n}(k) = X_{n-1}(k)+  \zeta_{n}(k) + \frac{k-1-\alpha}{k-1+\theta} X_{n-1}(k-1),\]
from which the recursion follows.
\end{proof}
%For every $n,k \in \N$ we denote by $F_{n,k}$ the family of functions
%\begin{equation}
%\begin{split}
%F_{n,1}(x_1,x_2,\cdots,x_n) &:= \sum_{j=1}^{n-1}\frac{\alpha x_j}{(j+\theta)\psi_{j+1}} \\
%F_{n,k}(x_k,x_2,\cdots,x_n) & := \frac{k-1-\alpha}{k-1+\theta} \sum_{j=k}^{n} x_j, \forall k> 1 \in \N.
%\end{split}
%\end{equation}
We may obtain an upper bound for $\theta_n(1)$ using the bounds for ratios of gamma functions in the Appendix:
\begin{equation}\label{ineq:boundthetan1}
\begin{split}
\theta_n(1) &=  1+ \displaystyle \sum_{j=1}^{n-1}\frac{\theta}{(j+\theta)\psi_{j+1}(1)} \\
&=  1+ \frac{\theta \Gamma(\alpha+\theta)}{\Gamma(1+\theta)} \displaystyle \sum_{j=1}^{n-1}\frac{\Gamma(j+\theta)}{\Gamma(j+\theta + \alpha)} \\
\theta_n(1)&\leq  1+ \frac{2\theta \Gamma(\alpha+\theta)}{(1-\alpha)\Gamma(1+\theta)} (n+\theta)^{1-\alpha}.
\end{split}
\end{equation}
This upper bound will be useful latter.

\subsection{The martingale component of $X_n(k)$}\label{sec:martingaleXn}
In this subsection we prove a concentration inequality result for a martingale sequence whose increment and quadratic variation satisfy certain hypothesis. Then we prove that the martingale component of $\{X_n(k)\}_{n}$ satisfies these conditions, for all $k$, proving then that the martingale component of the~$\{X_n(k)\}_{n}$ is well behaved.
%In this subsection we give upper bounds on the increments and on the quadratic variation of the martingale sequence~$\{M_n(k)\}_{n \in \N}$ for all $k$. We finally show, for all $k$ at once, that under these hypothesis $\{M_n(k)\}_{n \in \N}$ is not to large with high probability.

 \begin{lemma} \label{Martingale:Bound}
 	Let $d>0$ and $k \in \N$ be constants and $\{M_n\}_{n \in \N}$ be a martingale sequence satisfying
 	\begin{enumerate}
 		\item $|\Delta M_{j}| \leq \dfrac{d}{\Gamma(k+\theta)} \cdot (j-1+\theta)^{k-\alpha}$,
 		
 		\item $\displaystyle \mathbb{E}[(\Delta M_{j})^2|F_{j-1}] \leq \dfrac{d^2 \cdot (2k-\alpha) }{\Gamma(k+\theta)^2}  \cdot (j-1+\theta)^{2k-\alpha-1} \cdot \left(\frac{V_{j-1}}{\phi_{j-1}} + b_{j-1} \right)$,
 	\end{enumerate}
 	then there exists a constant $c_M$ such that
 	\begin{align*}
 	\Pp\left(|M_{n}-M_{m}| \geq \frac{\sqrt{2} d }{\Gamma(k+\theta)} (n+\theta)^{k-\frac{\alpha}{2}} A \right) \leq  e^{-c_M A}.
 	\end{align*}
 	for all $A \geq  \max_j \{b_j\}$.
 \end{lemma}
\begin{proof}
 Let $W_{n }$ the quadratic variation of the martingale $\{ M_{n} - M_m\}_n$. By our assumptions: 
 \begin{align*}
 W_n :&= \sum_{j=m+1}^{n} \mathbb{E}[(\Delta M_{j})^2| \F_j] \le \frac{d^2 \cdot |2k-\alpha|}{\Gamma(k+\theta)^2} \sum_{j=m+1}^{n}  (j-1+\theta)^{2k-\alpha-1} \left( \frac{V_{j-1}}{\phi_{j-1}} + b_{j-1} \right).
 \end{align*}
 Moreover, in the occurrence of the event $ \left\{ \sup_{j \in \N} \left( \frac{V_j}{\phi_j}\right) \leq A \right\},$  and using that $b_j \leq A$ we have
 \begin{align*}
 W_n &\leq  \frac{2d^2 \cdot A }{\Gamma(k+\theta)^2}  (n+\theta)^{2k-\alpha},
 \end{align*}
 in symbols, the following inclusion of events holds
 \begin{align*}
 \left\{ W_n \geq  \frac{2d^2 \cdot A }{\Gamma(k+\theta)^2}  (n+\theta)^{2k-\alpha} \right\} \subset \left\{ \sup_{j \in \N} \left( \frac{V_j}{\phi_j}\right) \geq A \right\},
 \end{align*}
which combined with Theorem \ref{Vm} yields
  \begin{align*}
 \Pp \left( W_n \geq  \frac{2d^2 \cdot A }{\Gamma(k+\theta)^2}  (n+\theta)^{2k-\alpha} \right) \leq \exp(-c_VA).
 \end{align*}
Finally, applying Lemma \ref{aux} with $R=  \frac{d^2 \cdot A }{\Gamma(k+\theta)^2}  (n+\theta)^{2k-\alpha} $ and $c_1=1$ we obtain
\begin{align*}
 	\Pp\left(|M_{n}-M_{m}| \geq \frac{\sqrt{2} d }{\Gamma(k+\theta)} (n+\theta)^{k-\frac{\alpha}{2}} A \right) &\leq \exp \left(\frac{-A}{2+\frac{2}{3}} \right) + \exp(-c_VA) \leq \exp(-c_M A)
 	\end{align*}
for some constant $c_M$.
\end{proof}

\begin{lemma} \label{ord:Mnk}
Let $\{M_n(k)\}_n$, $k \geq 1$, be the martingale defined in (\ref{def:Mn1}) and (\ref{def:Mnk}) and $A \geq 0$ a constant. Then there is a constant $h_{\alpha,\theta}$ such that
	\begin{align*}
 	\Pp\left(|M_{n}(k)| \geq \frac{ h_{\alpha, \theta} }{\Gamma(k+\theta)} (n+\theta)^{k-\frac{\alpha}{2}}  \left(A  + 2\log n \right)\right) \leq  \frac{e^{- A}}{n^2}.
 	\end{align*}
\end{lemma}
\begin{proof}
We will prove that the martingales in (\ref{def:Mn1}) and (\ref{def:Mnk}) satisfy the hypotheses of Lemma \ref{Martingale:Bound}, and the result will follow from that lemma.

Since $1/\psi_n(k)$ is increasing, by Lemma \ref{Bound:psik} in Appendix, the following bound holds
\begin{align*}
|\Delta M_{j}(1)|  \leq \frac{ e^{\frac{1}{12}} \Gamma(\alpha+\theta) }{\Gamma(1+\theta)} (n +\theta)^{1-\alpha}.
\end{align*}
And by definition of $\Delta M_j = \zeta_j$, we also have that
\begin{align*}
\mathbb{E}[(\Delta M_{j}(1))^2| \F_{j-1}] & =  \frac{1}{\psi_{j}(1)^2} \cdot \left[ \frac{\alpha V_{j-1} + \theta }{j-1+\theta} \cdot \left(1 - \frac{\alpha V_{j-1} + \theta- (1-\alpha) N_{j-1}(1)}{j-1+\theta}\right)^2 \right. \\
& \ \ + \left. \frac{(1 -\alpha)N_{j-1}(1) }{j-1+\theta} \cdot \left(-1 - \frac{\alpha V_{j-1} + \theta - (1-\alpha)N_{j-1}(1) }{j-1+\theta}\right)^2 \right]\\
& \leq  \frac{4}{\psi_{j}(1)^2} \frac{ V_{j-1} + \theta}{j-1+\theta}.
\end{align*}
Multiplying and deviding the above expression by $\phi_{j-1}$ and using the bound 
$$\frac{\phi_{j-1}}{(\psi_{j}(k))^2 \cdot(j-1+\theta)} \leq \frac{e^{1/6} \Gamma(1+\theta)\Gamma(\alpha+\theta)^2}{\Gamma(1+\theta+\alpha)\Gamma(k+\theta)^2}, (j-1+\theta)^{1-\alpha},$$
which may be deduced from see Lemma \ref{phipsi} in appendix, it follows that
\begin{align*}
\mathbb{E}[(\Delta M_{j}(1))^2| \F_{j-1}] & \leq  \frac{4\phi_{j-1}}{(j-1+\theta)\psi_{j}(1)^2} \frac{ V_{j-1} + \theta}{\phi_{j-1}} \\
& \leq   \frac{4 e^{1/6} \Gamma(1+\theta)\Gamma(\alpha+\theta)^2}{\Gamma(1+\theta+\alpha)\Gamma(1+\theta)^2} (j-1+\theta)^{1-\alpha} \frac{ V_{j-1} + \theta}{\phi_{j-1}},
\end{align*}
and since $2-\alpha > 1$, it also follows that
\begin{align*}
\mathbb{E}[(\Delta M_{j}(1))^2| \F_{j-1}] & \leq   \frac{4 (2-\alpha)  e^{1/6} \Gamma(1+\theta)\Gamma(\alpha+\theta)^2}{\Gamma(1+\theta+\alpha)\Gamma(1+\theta)^2} (j-1+\theta)^{1-\alpha} \frac{ V_{j-1} + \theta}{\phi_{j-1}}.
\end{align*}
Analogously, for $k>1$, we have
\begin{align*}
&  \mathbb{E}[(\Delta M_{j}(k))^2| \F_{j-1}] \\
&=   \frac{1}{\psi_{j}^2(k)} \cdot \left[ \frac{N_{j-1}(k-1)(k-1 - \alpha)}{j-1+\theta} \cdot \left(1- \frac{(k - 1- \alpha)N_{j-1}(k-1) - (k - \alpha)N_n(k)}{j-1+\theta} \right)^2  \right. \\
&\ \ + \left. \frac{N_{j-1}(k)(k -\alpha) }{j-1+\theta} \cdot \left(-1- \frac{(k - 1- \alpha)N_{j-1}(k-1) - (k - \alpha)N_{j-1}(k)}{j-1+\theta} \right)^2   \right]\\
&\leq    \frac{4}{\psi_{j}^2(k)(j-1+\theta)} [ N_{j-1}(k-1)(k-1 - \alpha)+ N_{j-1}(k)(k -\alpha) ].
\end{align*}
Since $N_j(k)$ is bounded from above by $V_j$, for all $k$ and $j$, we obtain 
\begin{align*}
 \mathbb{E}[(\Delta M_{j}(k))^2| \F_{j-1}] &\leq   \frac{4 \phi_{j-1}}{(\psi_{j}(k))^2 \cdot(j-1+\theta)} \left[ \frac{V_{j-1}}{\phi_{j-1}} \cdot(k-1 - \alpha)+ \frac{V_{j-1}}{\phi_{j-1}}\cdot (k -\alpha) \right]  \\
&  \leq  \frac{4(2k-\alpha)\Gamma(1+\theta)\Gamma(\alpha+\theta)^2}{\Gamma(1+\theta+\alpha)\Gamma(k+\theta)^2}   (j-1+\theta)^{2k-\alpha-1} \frac{V_{j-1}}{\phi_{j-1}}.
\end{align*}
Finally, by Lemma \ref{Martingale:Bound} we have, for 
\[
h_{\alpha,\theta} = \frac{2 \sqrt{2} e^{\frac{1}{12}} }{c_M} \cdot \Gamma(\alpha+\theta) \cdot \max \left\{1, 2\sqrt{\frac{\Gamma(1+\theta)}{\Gamma(1+\theta+\alpha)}} \right\},
\]
that
\begin{align*}
 	\Pp\left(|M_{n}(k)| \geq \frac{ h_{\alpha,\theta} }{\Gamma(k+\theta)} (n+\theta)^{k-\frac{\alpha}{2}} (A+ \log^2 n)  \right) \leq  \frac{e^{- A}}{n^2},
\end{align*}
as we desired.
\end{proof}

\section{Bounds for the number of parts with size $k$} \label{sec:UperXnk}

Let us go through what we did in Section \ref{numbersizek}. In Subsection \ref{sec:recurrenceXn} we found recurrence relations relating the values $X_n(k)=N_n(k)/\psi_n(k)$ for different $k$ and $n$. This is the content of Lemma \ref{lemma:recX}, where we obtained that:
\begin{eqnarray*}
X_{n}(1) &=& M_{n}(1) + \sum_{j=1}^{n-1}\frac{\alpha V_j }{(j+\theta)\psi_{j+1}(1)}+ \theta_{n}(1); \\
X_{n}(k) & =& M_{n}(k)+X_k(k) + \frac{k-1-\alpha}{k-1+\theta} \sum_{j=k}^{n-1} X_j(k-1), \forall k > 1.\end{eqnarray*}
The terms $M_n(k)$ above are martingales. Subsection \ref{sec:martingaleXn} proves that the martingale terms are all small. Since we already know $V_j/\phi_j\approx V_*$ for $j$ large, this will lead to bounds of the form:
\begin{eqnarray*}
X_{n}(1) &\approx & a_0(1)\,V_*; \\
X_{n}(k) & \approx &\frac{k-1-\alpha}{k-1+\theta} \sum_{j=k}^{n-1} X_j(k-1), \forall k > 1\end{eqnarray*}
where 
\[a_0(1):=\frac{\alpha}{\alpha+\theta}.\]
If we treat the above recursions as equalties, we then obtain by induction in $k$ that
\[X_{n}(k) \approx a_0(k)\,V_*\,n^{k}\]
where 
\[a_0(k) = \frac{(k-1-\alpha) \cdot a_0(k-1)}{(k-1+\theta)k} = \frac{\Gamma(k-\alpha) \Gamma(1+\theta)}{k! \cdot \Gamma(1-\alpha) \Gamma(k+\theta)} a_0(1).\]

The purpose of this section is to make the above approximations precise and to show that $X_n(k)$ does behave as expected up to leading order, in high probability. In particular, we will prove the following Theorem (recall the definition of $X_n(k)$ in (\ref{def:Xnk})).

\begin{theorem}\label{UpXn}
Given $A>K(\alpha,\theta)$, $n \in \N$ and $k \leq n$, there are coefficients $a_0(k)$ (defined above) and $a_1(k)$ with
$a_1(k) = O \left( a_0(k) \cdot k^{\alpha+2} \right) $, such that the following holds. Define the event where $X_m(s)$ is ``well-controlled from above".
\[F^{(up)}_{m,s}:=\Big\{ X_m(s) \leq a_0(s)V_* (m-1)^s + a_1(s) (m+\theta)^{s-\alpha/2} (A + \log n)\Big\}.\]
Similarly, define the event that $X_m(s)$ is ``well-controlled from below".
\[F^{(dn)}_{m,s}:=\Big\{ X_m(s) \geq a_0(s)V_* (m-s)^s - a_1(s) (m+\theta)^{s-\alpha/2} (A + \log n) \Big\}.\]
Finally, define the event where the above inequalities hold for all times $m\leq n$ and part sizes $s\leq k$:
\[E_{n,k} := \bigcap_{m\leq n}\bigcap_{s\leq k}\,(F^{(up)}_{m,s}\cap F^{(dn)}_{m,s}).\]
Then:$$ \Pp( E_{n,k} ) \geq 1 - \frac{k}{n} e^{-A}.$$
\end{theorem}
As we will see, this theorem follows directly from the results in the remainder of this section. 

\begin{proof}[Proof of Theorem \ref{UpXn}]
 The bound $a_1(k) = O \left( a_0(k) \cdot k^{\alpha+2} \right) $ is contained in Lemma \ref{BoundRecCoef} in subsection \ref{sec:choice}. The probability of $E_{n,k}$ is bounded in Lemmas \ref{UpX1} and \ref{upXn} in subsection \ref{sec:proofbyinductionXn}.
\end{proof}

\subsection{The choice of coefficients}\label{sec:choice}

The coefficients $a_0(1)$ and $a_1(1)$ will arise from the analysis of the recursion (\ref{rec:Xn1}) in the Lemma \ref{lemma:recX}. As we have seen, $a_0(k)$ appears naturally when we work out the leading order terms for $X_n(k)$. The extra coefficient $a_1(k)$ controls the error, and comes from combining errors in estinating $X_s(k-1)$ (induction step); the error in setting $M_n(k)\approx 0$; and various other estimates in the proof (see (\ref{rec:Xnk}) and Lemma \ref{BoundRecCoef}). 

We define:
\begin{align}
\label{a01}a_0(1)&:= \frac{\alpha}{\alpha+\theta}, \\
\label{a11}a_1(1)&:= \frac{ h_{\alpha, \theta} }{\Gamma(1+\theta)}+ \frac{\alpha}{c_V(\alpha+\theta)(1-\frac{\alpha}{2})} + 1+ \frac{2\theta \Gamma(\alpha+\theta)}{(1-\alpha)\Gamma(1+\theta)}, \\
\label{a0} a_0(k) &:= \frac{(k-1-\alpha) \cdot a_0(k-1)}{(k-1+\theta)k} = \frac{\Gamma(k-\alpha) \Gamma(1+\theta)}{k! \cdot \Gamma(1-\alpha) \Gamma(k+\theta)} a_0(1) ,   \\
\label{a1} a_1(k) &:= \frac{h_{\alpha,\theta} }{ \Gamma(k+\theta) } +  \frac{(k-1-\alpha) \cdot a_1(k-1)}{(k-1+\theta)(k-\frac{\alpha}{2})}. 
\end{align}
From the analysis of recursions involving $X_n(k)$, it will arise naturally terms which are polynomials whose coefficients are the above coefficients. Thus, it will be useful to have estimates for such polynomials as well. We do this in the next lemma.
\begin{lemma} \label{BoundRecCoef}
The coefficients $a_0(k)$ and $a_1(k)$ defined as in (\ref{a0}) and (\ref{a1}) satisfy the following relations:

\begin{enumerate}

\item $\displaystyle \frac{k-1-\alpha}{k-1+\theta} \sum_{j=k}^{m-1}   a_0(k-1) j^{k-1} \ \ \leq a_0(k) m^k$,

\item $\displaystyle \frac{k-1-\alpha}{k-1+\theta} \sum_{j=k}^{m-1}   a_0(k-1) (j-(k-1))^{k-1} \ \ \geq a_0(k) (m-k)^k$,

\item $\displaystyle \frac{k-1-\alpha}{k-1+\theta} \sum_{j=k}^{m-1}  a_1(k-1) (j+\theta)^{k-1-\alpha/2}  \leq \left( a_1(k)- \frac{h_{\alpha,\theta} }{ \Gamma(k+\theta) }  \right)(m+\theta)^{k-\alpha/2}$,

\item $a_1(k) \leq C_U a_0(k) \cdot k^{\alpha+2} $, for some constant $C_U$.
\end{enumerate}
\end{lemma}

\begin{proof} Throughout this proof we will make use of the integral bound below
\begin{equation}\label{eq:integralbound}
\frac{(m-k)^k}{k}= \int_{0}^{m-k} x^{k-1} \le \sum_{j=1}^{m-1}   j^{k-1} \leq \int_{0}^{m} x^{k-1} = \frac{m^k}{k}.
\end{equation}
\textit{(1)} For the first bound, observe that
\begin{align*}
 \frac{k-1-\alpha}{k-1+\theta} \sum_{j=k}^{m-1}   a_0(k-1) j^{k-1} \leq \frac{k-1-\alpha}{k-1+\theta}  a_0(k-1) \sum_{j=1}^{m-1}   j^{k-1}.
\end{align*}
Using the upper bound given by (\ref{eq:integralbound}), yields
\begin{align*}
 \frac{k-1-\alpha}{k-1+\theta} \sum_{j=\ell+1}^{m-1}   a_0(k-1) j^{k-1} \leq \frac{(k-1-\alpha) \cdot a_0(k-1)}{(k-1+\theta)k}
\end{align*}
which is exactly the definition of $a_0(k)$ in (\ref{a0}).
\\ 

\noindent \textit{(2)} For the second relation, we have
\begin{align*}
 \frac{k-1-\alpha}{k-1+\theta} \sum_{j=k}^{m-1}   a_0(k-1) (j-(k-1))^{k-1} &=  \frac{k-1-\alpha}{k-1+\theta} \sum_{j=1}^{m-k}   a_0(k-1) j^{k-1} \\
 \geq   \frac{k-1-\alpha}{k-1+\theta}  a_0(k-1) \sum_{j=1}^{m-k}   j^{k-1}.
\end{align*}
By the lower bound given by (\ref{eq:integralbound}), we obtain
\begin{align*}
 \frac{k-1-\alpha}{k-1+\theta} \sum_{j=1}^{m-k}   a_0(k-1) j^{k-1} &\geq \frac{(k-1-\alpha) \cdot a_0(k-1)}{(k-1+\theta)k}(m-k)^k = a_0(k)(m-k)^k.
\end{align*}

\noindent \textit{(3)} If we proceed exactly as in the item (1) we obtain
\begin{align*}
\displaystyle \frac{k-1-\alpha}{k-1+\theta} \sum_{j=k}^{m-1}  a_1(k-1) (j+\theta)^{k-1-\alpha/2} 
\leq \frac{(k-1-\alpha) \cdot a_1(k-1)}{(k-1+\theta)(k-\frac{\alpha}{2})}(m+\theta)^{k-\alpha/2},
\end{align*}
but by definition (\ref{a1})
\begin{align*}
\frac{(k-1-\alpha) \cdot a_1(k-1)}{(k-1+\theta)(k-\frac{\alpha}{2})}  = \left( a_1(k)- \frac{h_{\alpha,\theta} }{ \Gamma(k+\theta) }  \right).\\
\end{align*}

\noindent \textit{(4)} We begin substituting the formulas for $a_0(k)$ and $a_1(0)$ in an analogous way we did above, to obtain an affine recurrence
\begin{align} \label{eq:a1a0}
\frac{a_1(k)}{a_0(k)} &=  d \cdot \frac{k!}{ \Gamma(k-\alpha) } +  \frac{k \cdot a_1(k-1)}{(k-\frac{\alpha}{2}) \cdot a_0(k-1)},
\end{align}
where $d$ is defined as 
\[
d := d_{\alpha,\theta} \cdot \frac{ \Gamma(1-\alpha) }{ a_0(1) \cdot \Gamma(1+\theta) }.
\]
We rearrange (\ref{eq:a1a0}) by letting $s(k)$ to be 
\[
s(k) := \frac{a_1(k)}{a_0(k)} \frac{\Gamma(k-\frac{\alpha}{2}+1 )}{\Gamma(k+1)}
\]
and multiplying both sides by $\frac{\Gamma(k-\frac{\alpha}{2}+1 )}{\Gamma(k+1)}$ to obtain the identity below
\begin{align*}
s(k) &= s(k-1) +   d \cdot \frac{\Gamma(k-\frac{\alpha}{2}+1) }{ \Gamma(k-\alpha) },
\end{align*}
so we can find the general formula to the recurrence
\begin{align*}
s(k) &= s(1) +   d \cdot \sum_{j=1}^{k} \frac{ \Gamma(j-\frac{\alpha}{2}+1) }{ \Gamma(j-\alpha) }.
\end{align*}
Using the bound $\frac{\Gamma(j-\frac{\alpha}{2}+1)}{\Gamma(j-\alpha)} \leq e^{1/12} j^{1+ \frac{\alpha}{2}}$ we have
\begin{align*}
s(k) &= s(1) +  e^{1/12} d \cdot \sum_{j=1}^{k}  j^{\frac{\alpha}{2}+1} \leq s(1) +  e^{1/12} d \cdot k^{\frac{\alpha}{2} + 2} \leq d_1 k^{\frac{\alpha}{2} + 2},
\end{align*}
where $d_1 = 2\max\{s(1), e^{1/12}d\}$. Finally, we obtain
\begin{align*}
 \frac{a_1(k)}{a_0(k)} &\leq d_1 \dfrac{\Gamma(k+1)}{\Gamma(k-\frac{\alpha}{2}+1 )} k^{\frac{\alpha}{2} + 2} = C_U k^{\alpha + 2},
\end{align*}
for some $C_U$.

\end{proof}

\subsection{Bound on $X_n(k)$}\label{sec:proofbyinductionXn}
We now bound the probability of the events $E_{n,k}$ defined in the statement of Theorem \ref{UpXn}. Our approach is induction on $k$. But before we go to the proof, let us recall the definition of the sequence of events $E_{n,k}$. The event $F^{(up)}_{m,s}$ is defined as the event where $X_m(s)$ is ``well-controlled from above"
\[
F^{(up)}_{m,s} =\Big\{ X_m(s) \leq a_0(s)V_* (m-1)^s + a_1(s) (m+\theta)^{s-\alpha/2} (A + \log n)\Big\}.\]
Analogously, $F^{(dn)}_{m,s}$ is the event where $X_m(s)$ is ``well-controlled from below"
\[
F^{(dn)}_{m,s}:=\Big\{ X_m(s) \geq a_0(s)V_* (m-s)^s - a_1(s) (m+\theta)^{s-\alpha/2} (A + \log n) \Big\}.
\]
Finally, the event $E_{n,k}$ is the event where the above inequalities hold for all times $m\leq n$ and part sizes $s\leq k$:
\[E_{n,k} := \bigcap_{m\leq n}\bigcap_{s\leq k}\,(F^{(up)}_{m,s}\cap F^{(dn)}_{m,s}).\]
%\subsubsection{First step of the induction}
Now, we start by the case $k=1$.
\begin{lemma}[Case $k=1$] \label{UpX1}
Given $A>0$ and $n \in \N$, let $E_{n,1}$ be as in the statement of Theorem \ref{UpXn}. Then:
%\begin{align*}
%E_{n,1} = \Big\{ |X_j(1) - a_0(1) V_*  (j-1)| \leq     a_1(1) (j+\theta)^{1-\alpha/2} (A + \log n)  \ \ \textrm{for all $1 \leq j \leq n$} \Big\} 
%\end{align*}
\begin{align*}
\Pp(E_{n,1}) \geq 1 - \frac{e^{-A}}{n}.
\end{align*}
\end{lemma}
\begin{proof}
The equation (\ref{rec:Xn1}) says us that
\begin{align*}
X_n(1) &= M_{n}(1) + \sum_{j=1}^{n-1}\frac{\alpha V_j }{(j+\theta)\psi_{j+1}}+ \theta_{n}(1).
\end{align*}	
We will bound each term in the right-hand side to obtain a bound on $X_n(1)$. Before, we manipulate algebraically the above expression for $X_n(1)$ in such way it can be expressed in terms of the observables we already know how to control. In this direction, we start summing and subtracting the sum below 
\[
\sum_{j= 1}^{n-1} \frac{\alpha \phi_j }{(j+\theta)\psi_{j+1}} V_*
\]
in the second member of (\ref{rec:Xn1}) to use that the ratio $V_j/\phi_j$ is approximated by $V_*$. This yields
\begin{align}\label{eq:Xn1mod}
X_n(1) &= M_{n}(1) + \sum_{j= 1}^{n-1} \frac{\alpha \phi_j }{(j+\theta)\psi_{j+1}} \left( \frac{V_j}{\phi_j} -V_*  \right)+ \sum_{j= 1}^{n-1} \frac{\alpha \phi_j }{(j+\theta)\psi_{j+1}} V_*  + \theta_{n}(1).
\end{align}
Using the relation below 
\[
\dfrac{ \phi_j }{(j+\theta)\psi_{j+1}(1)} =  \dfrac{1}{(\theta+\alpha)}
\] 
on identity (\ref{eq:Xn1mod}) allows us to obtain
\begin{align*}	
%X_n(1)	&=  M_{n}(1) + \sum_{j= 1}^{n-1}  \frac{\alpha}{(\theta +\alpha)} \left( \frac{V_j}{\phi_j} -V_*  \right) + \sum_{j= 1}^{n-1}  \frac{\alpha}{(\theta +\alpha)} V_*  + \theta_{n}.
	X_n(1) &=  M_{n}(1) + \sum_{j= 1}^{n-1}  \frac{\alpha}{(\theta +\alpha)} \left( \frac{V_j}{\phi_j} -V_*  \right) + \frac{\alpha V_* }{(\theta+\alpha)} (n-1) + \theta_{n}(1).
%\Rightarrow |X_n(1)| &= \left| M_{n}(1) + \sum_{j= 1}^{n-1}  \frac{\alpha}{(\theta +\alpha)} \left( \frac{V_j}{\phi_j} -V_*  \right) + \frac{\alpha V_* }{(\theta+\alpha)} (n-1) + \theta_{n} \right|.
\end{align*}
Taking the absolute value on both sides of the above identity and using the triangle inequality yields
\begin{align}	
 |X_n(1)| &\leq |M_{n}(1)| + \left|\sum_{j= 1}^{n-1}  \frac{\alpha}{(\theta +\alpha)} \left( \frac{V_j}{\phi_j} -V_*  \right)  \right| + \frac{\alpha V_* }{(\theta+\alpha)} (n-1) + \theta_{n}(1).
\end{align}
and
\begin{align}
|X_n(1)| & \geq \frac{\alpha V_* }{(\theta+\alpha)} (n-1) - |M_{n}(1)| - \left| \sum_{j= 1}^{n-1}  \frac{\alpha}{(\theta +\alpha)} \left( \frac{V_j}{\phi_j} -V_*  \right) \right|.
\end{align}
By Lemma \ref{ord:Mnk}, the probability of the event below
\begin{equation} \label{eq:evMn1}
\left \lbrace |M_{n}(1)| \geq \frac{ h_{\alpha, \theta} }{\Gamma(1+\theta)} (n+\theta)^{1-\frac{\alpha}{2}} (A + \log n) \right \rbrace
\end{equation}
is bounded from above by
\begin{align}
\label{step:martingale} \Pp\left(|M_{n}(1)| \geq \frac{ h_{\alpha, \theta} }{\Gamma(1+\theta)} (n+\theta)^{1-\frac{\alpha}{2}} (A + \log n) \right) \leq \frac{e^{-A}}{n^2},
\end{align}
and by Corollary \ref{thm:V}, with $\delta = \frac{e^{-A}}{n^2}$ and observing that $\log m \leq \log n$, for $1 \leq m \leq n-1$ we have 
\begin{align}
\Pp \left( \left| \frac{V_j}{\phi_j} - V_* \right| \geq \frac{A + \log n}{c_V (j+\theta)^{\alpha/2}},\mbox{ for some } 1 \leq j \leq n-1  \right) \leq \frac{e^{-A}}{n^2}.
\end{align}
On the occurrence of the event 
\begin{equation} \label{eq:evV}
\left\{ \left| \frac{V_j}{\phi_j} - V_* \right| \leq \frac{A +  \log n}{c_V (m+\theta)^{\alpha/2}}, \mbox{ for some } 1 \leq j \leq n-1  \right\}
\end{equation} 
we have
\begin{align}
\left| \sum_{j= 1}^{n-1}  \frac{\alpha}{(\theta +\alpha)} \left( \frac{V_j}{\phi_j} -V_*  \right) \right| &\leq \sum_{j= 1}^{n-1}  \frac{\alpha}{(\theta +\alpha)} \frac{A +  \log n}{c_V(j+\theta)^{\alpha/2}}\\
\label{step:V*}  &\leq \frac{(n-1+\theta)^{1-\alpha/2}}{1-\alpha/2}  \frac{\alpha (A +  \log n)}{c_V(\theta+\alpha)}.
\end{align}
By (\ref{ineq:boundthetan1}), the term $\theta_n$(1) is bounded in the following way
\begin{align*}
\theta_n(1) &\leq  1+ \frac{2\theta \Gamma(\alpha+\theta)}{(1-\alpha)\Gamma(1+\theta)} (n+\theta)^{1-\alpha} \\
&\leq \left( 1+ \frac{2\theta \Gamma(\alpha+\theta)}{(1-\alpha)\Gamma(1+\theta)} \right) (n+\theta)^{1-\frac{\alpha}{2}}(A+ \log n).
\end{align*}
Thus, on the occurrence of both events (\ref{eq:evMn1}) and (\ref{eq:evV}), we have
$$|X_n(1)| \leq a_0(1) V_*  (j-1)  +  a_1(1) (j+\theta)^{1-\alpha/2} (A + \log n)$$
for $a_0(1)$ and $a_1(1)$ whose definition we recall below
\begin{align}
a_0(1) &= \frac{\alpha}{\alpha+\theta}, \\
a_1(1) &= \frac{ h_{\alpha, \theta} }{\Gamma(1+\theta)}+ \frac{\alpha}{c_V(\alpha+\theta)(1-\frac{\alpha}{2})} + 1+ \frac{2\theta \Gamma(\alpha+\theta)}{(1-\alpha)\Gamma(1+\theta)}.
\end{align}
Therefore
\begin{align*}
\Pp((E_{n,1})^c)& \leq \sum_{j=1}^n \Pp\left(|M_{n}(j)| \geq \frac{ h_{\alpha, \theta} }{\Gamma(j+\theta)} (n+\theta)^{j-\frac{\alpha}{2}} (A + \log n) \right) %\\
%&\leq n  \frac{e^{-A}}{n^2} \\
\leq  \frac{e^{-A}}{n},
\end{align*}
which proves the first step of the induction.
\end{proof}
Now we prove on the next lemma the inductive step.
\begin{lemma}[The inductive step] \label{upXn}
Given $A>0$, $n \in \N$ and $k \leq n$, let $E_{n,k}$ be as in the statement of Theorem \ref{UpXn}. Then:
$$ \Pp( E_{n,k} ) \geq 1 - \frac{k}{n}e^{-A}.$$
\end{lemma}
\begin{proof} The key step of the proof is the following inclusion of events 
\begin{equation}\label{eq:inductivealgumacoisa}
E_{n,k} \supset E_{n,k-1} \cap \left\{|M_{j}(k)| \leq \frac{ h_{\alpha, \theta} }{\Gamma(k+\theta)} (j+\theta)^{k-\frac{\alpha}{2}} (A +  \log n)\mbox{, for all }1\leq j\leq n\right\},
\end{equation}	
for all $k \geq 2$. The result then follows by induction from our previous results and the inequality below 
\begin{align*}
\Pp( E_{n,k}^c ) &\leq \Pp( E_{n,k-1}^c ) + \sum_{j=1}^{n} \Pp\left(|M_{j}(k)| \geq \frac{ h_{\alpha, \theta} }{\Gamma(k+\theta)} (j+\theta)^{k-\frac{\alpha}{2}}(A +  \log n) \right)\\
& \leq (k-1) \frac{e^{-A}}{n} + n \frac{e^{-A}}{n^2}.
%= k \frac{e^{-A}}{n}.
\end{align*}
Let us then explain why (\ref{eq:inductivealgumacoisa}) holds.
At a high level, when event $E_{n,k-1}$ occurs,we have that all $X_j(s)$ are ``well-behaved" for all values of $ s \leq k-1$ and $1 \leq j \leq n$. Now we will prove that combining this with bounds on the martingale component of $X_n(k)$, $X_n(k)$ itself will be ``well-behaved". To do that, we will just bound the recursion for $X_n(k)$ using the bounds given by the above events and Lemma \ref{BoundRecCoef}.

We start by restating the recursion for a fixed $m$: 
\begin{align}\label{eq:recXmk}
X_{m}(k) &=  M_{m}(k) + X_k(k)+ \frac{k-1-\alpha}{k-1+\theta} \sum_{j=k}^{m-1} X_j(k-1).
\end{align}
We let $Z_m(k-1)$ denote the sum in the RHS of the previous display.
\begin{equation}
Z_m(k-1) := \frac{k-1-\alpha}{k-1+\theta} \sum_{j=k}^{m-1} X_j(k-1).
\end{equation}
In the event
\[
E_{n,k-1} \cap  \left\{|M_{j}(k)| \leq \frac{ h_{\alpha, \theta} }{\Gamma(k+\theta)} (j+\theta)^{k-\frac{\alpha}{2}} (A + \log n),  \ \ \textrm{for all $1 \leq j \leq n$} \right\}
\]
we have that each $X_j(k-1)$ is bounded by 
\begin{align*}
X_j(k-1) \leq  a_0(k-1)V_* j^{k-1} + a_1(k-1) (j+\theta)^{k-1-\alpha/2} (A + \log n),
\end{align*}
which implies the following bound
\begin{equation}
Z_m(k-1)  \leq \frac{k-1-\alpha}{k-1+\theta} \sum_{j=k}^{m-1}  \Big[  a_0(k-1)V_* j^{k-1} + a_1(k-1) (j+\theta)^{k-1-\alpha/2} (A + \log n) \Big].
\end{equation}
Now, recall that Lemma \ref{BoundRecCoef} gives us bounds on the polynomials on $j$ whose coefficients are $a_i(k-1)$. Thus, combining this with the above bound we obtain
\begin{align}\label{aux:upper} 
Z_m(k-1) \leq a_0(k)V_* m^{k} & + \left( a_1(k)- \frac{h_{\alpha,\theta} }{ \Gamma(k+\theta) }  \right)(m+\theta)^{k-\alpha/2} (A +  \log n).
\end{align}
Arguing the same way, but applying the lower bound to $X_j(k-1)$ given by $E_{n,k-1}$ instead we may obtain that $Z_m(k-1)$ is bounded from below by
\begin{align*}
%&\frac{k-1-\alpha}{k-1+\theta} \sum_{j=k}^{m-1} X_j(k-1)
 \frac{k-1-\alpha}{k-1+\theta}  \sum_{j=k}^{m-1} \left[  a_0(k-1)V_*(j-(k-1))^{k-1}  - a_1(k-1) (A + 2 \log n)(j+\theta)^{k-1-\alpha/2} \right].
\end{align*}
And again, by Lemma \ref{BoundRecCoef} we have
\begin{align}
\label{aux:lower} Z_m(k-1) \geq    a_0(k)V_*(j-k)^{k}   - \left( a_1(k)- \frac{h_{\alpha,\theta} }{ \Gamma(k+\theta) }  \right)(m+\theta)^{k-\alpha/2} (A + 2 \log n).
\end{align}
On the other hand, since, for all $k$ we also have
\begin{align*}
X_{m}(k) &\geq M_{m}(k)  +\frac{k-1-\alpha}{k-1+\theta} \sum_{j=k}^{m-1} X_j(k-1),
\end{align*}
the result then follows by joining (\ref{aux:upper}) and (\ref{aux:lower}) with the martingale bound given by the other event in the intersection.
\end{proof}
\section{Proof of Theorem \ref{thm:main}} \label{proofmain}
This section is devote to the proof of Theorem \ref{thm:main} which ensures bounds to the number of parts of size $k$ itself.  
\begin{proof}[Proof of Theorem \ref{thm:main}]
First observe that in the event $E_{n,k}$ we have
\begin{align*}
X_n(k) - a_0(k)V_* n^k \leq  a_1(k) (n+\theta)^{k-\alpha/2} (A + \log n).
\end{align*}
Consequently, by lemma \ref{BoundRecCoef}
\begin{align*}
X_n(k) - a_0(k)V_* n^k \leq  C_U a_0(k) \cdot k^{\alpha+2} (n+\theta)^{k-\alpha/2} (A + \log n).
\end{align*}
By the same argument we also have the lower bound
\begin{align*}
%X_n(k) &\geq a_0(k)V_* (n-k)^k  -a_1(k)(n+\theta)^{k-\alpha/2} (A + \log n) \\
X_n(k) \geq a_0(k)V_* (n-k)^k  - C_U a_0(k)(n+\theta)^{k-\alpha/2} (A + \log n).
\end{align*}
Moreover, note that 
\begin{align*}
(n-k)^k =n^k \left(1-\frac{k}{n} \right)^k .
\end{align*}
By Bernoulli's inequality,
\[
(1+x)^m \geq 1+mx
\] 
for all $x \geq -1$ and $m \in \N$.  Then, for $x= - k/n \geq -1$ and $m = k$ we have
\begin{align*}
n^k \left(1-\frac{k}{n} \right)^k  &\geq n^k\left(1-\frac{k^2}{n} \right) = n^k - k^2n^{k-1}.
\end{align*}
Thus
\begin{align*}
X_n(k) &\geq a_0(k)V_* (n^k - k^2n^{k-1})  - C_U a_0(k)(n+\theta)^{k-\alpha/2} (A + \log n),
\end{align*}
which implies
\begin{equation*}
X_n(k) -  a_0(k)V_* n^k \geq  - a_0(k)V_* k^2 n^{k-1} - C_U a_0(k) \cdot k^{\alpha+2} (n+\theta)^{k-\alpha/2} (A + \log n).
\end{equation*}
Moreover, on the occurrence of the event  
\[
E_*:= \left\{ V_* \leq \frac{2(A+\log n)}{c_V} \right\}
\] 
we also have
\begin{align*}
X_n(k) -  a_0(k)V_* n^k &\geq  - a_0(k)  \left( \frac{2}{c_V} k^2 n^{k-1} + C_U \cdot k^{\alpha+2} (n+\theta)^{k-\alpha/2} \right)(A + \log n) \\
&\geq - a_0(k)  D \cdot k^{\alpha+2} (n+\theta)^{k-\alpha/2} (A + \log n),
\end{align*}
where $D :=  2c^{-1}_V + C_U $. Thus, on the intersection of $E_{n,k}$ and $E_*$, we have 
%$E_{n,k} \cap \left\{ V_* \leq \frac{2(A+\log n)}{c_V} \right\}$ 
\begin{align}
\label{eq:D}|X_n(k) -  a_0(k)V_* n^k| &\leq  D a_0(k)  k^{\alpha+2} (n+\theta)^{k-\alpha/2}(A + \log n).
\end{align}
To simplify our writing, define 
\begin{equation}
f_n(k) := a_0(k) \cdot \psi_n(k) \cdot n^k.
\end{equation}
Multiplying both sides of (\ref{eq:D}) by $\psi_{n}(k)$ we have
\begin{align*}
|N_n(k) -  f_n(k)V_*| &\leq  D f_n(k)  k^{\alpha+2} \frac{(n+\theta)^{k-\alpha/2}}{n^k}(A + \log n).
\end{align*}
Now, using that $1+x \leq e^x$, we have
\begin{align}
\label{bound:exponential} (n+\theta)^{\gamma}  \leq e^{\frac{\theta \gamma}{n}} n^{\gamma},
\end{align}
which implies, for $k< n/\theta$
\begin{align*}
|N_n(k) - f_n(k) V_*| &\leq  eD f_n(k)  \frac{ k^{\alpha+2}}{n^{\alpha/2}}  (A + \log n).
\end{align*}
Recalling the definition of $a_0(k)$ 
\begin{equation}
a_0(k) = \frac{\Gamma(k-\alpha) \Gamma(1+\theta)}{k! \cdot \Gamma(1-\alpha) \Gamma(k+\theta)} a_0(1)
\end{equation}
and replacing it and $\psi_n(k)$ on $f_n(k)$ it may be written as
\begin{align*}
f_n(k) &= \left[ \frac{ \alpha \Gamma(1+\theta)}{ \Gamma(1-\alpha)\Gamma(\alpha+\theta+1) }  \frac{ \Gamma(k-\alpha)}{ \Gamma(k+1)} \right] \cdot \left[  \frac{\Gamma(n-k+\alpha+\theta)}{\Gamma(n+\theta)} \right] n^k.
\end{align*}
By Lemma \ref{gammagamma} in the Appendix, for $k$ of order $n^{\alpha/(2\alpha+4)}$, we have
\begin{equation}
\left[  \frac{\Gamma(n-k+\alpha+\theta)}{\Gamma(n+\theta)} \right] = \frac{1}{n^{k-\alpha}}  \left(1 + O\left(\frac{k^2}{n-k} \right) \right),
\end{equation}
which implies
\begin{align}
\label{eq:orderfnk}f_n(k) &= \left[ \frac{ \alpha \Gamma(1+\theta)}{ \Gamma(1-\alpha)\Gamma(\alpha+\theta+1) }  \frac{ \Gamma(k-\alpha)}{ \Gamma(k+1)} \right]  \left(1 + O\left(\frac{k^2}{n} \right) \right) n^{\alpha}.
\end{align}
Now, by the above identity, we have that
\begin{equation}\label{eq:nnk}
N_n(k_n) - c_{\alpha,\theta}  \frac{ \Gamma(k-\alpha)}{ \Gamma(k+1)}   \cdot n^{\alpha} \cdot V_* = N_n(k_n) - f_n(k)V_* + \frac{ \Gamma(k-\alpha)}{ \Gamma(k+1)}  O\left(\frac{k^2}{n} \right)  \cdot n^{\alpha} \cdot V_*
\end{equation}
Also, observe that
\begin{equation}\label{ineq:gammaka}
\frac{\Gamma(k-\alpha)}{\Gamma(k+1)} \leq  e^{\frac{1}{12}} \left(1+\frac{1+\alpha}{k-\alpha} \right)^{1/2}\left( \frac{1}{k-\alpha}\right)^{1+\alpha}
\leq \frac{4}{k^{1+\alpha}}.
\end{equation}
Applying the triangle inequality on (\ref{eq:nnk}), recalling we are inside $E_*$ and using the above upper bound, we obtain that
\begin{align*}
\left| N_n(k) -  c_{\alpha,\theta}  \frac{ \Gamma(k-\alpha)}{ \Gamma(k+1)}   \cdot n^{\alpha} \cdot V_* \right| \leq  D_2 \left( \frac{ k^{\alpha+2}}{n^{\alpha/2}}  f_n(k)  +  \left( \frac{k  }{n} \right)^{1-\alpha}  \right) (A +  \log n).
\end{align*}
for some positive constant $D_2$. Finally, for every $k$ satisfying
\[
k \leq \frac{\varepsilon n^{\frac{\alpha}{2\alpha+4}}}{(\log n)^{\frac{1}{\alpha+2}}}
\]  there is another absolute constant $C$ such that 
\begin{align*}
\left| N_n(k) -  c_{\alpha,\theta}  \frac{ \Gamma(k-\alpha)}{ \Gamma(k+1)}   \cdot n^{\alpha} \cdot V_* \right|   &\leq  C  \frac{ \Gamma(k-\alpha)}{ \Gamma(k+1)}   \cdot n^{\alpha}   \frac{ \left( \frac{\varepsilon n^{\frac{\alpha}{2\alpha+4}}}{(\log n)^{\frac{1}{\alpha+2}}} \right)^{\alpha+2}}{n^{\alpha/2}}  (A + \log n) \\
&\leq  C \frac{ \Gamma(k-\alpha)}{ \Gamma(k+1)}   \cdot n^{\alpha} \cdot \varepsilon^{\alpha+2} \cdot \left( \frac{A}{\log n} + 1 \right),
\end{align*}
proving our main theorem.
\end{proof}

\section{Final remarks}
The main open problem that could be addressed by our methods is to push the analysis to larger values of $k$. We conjecture that a tighter analysis would work for all $k=o(n^{\alpha/(1+\alpha)})$ or some similar range. This is in the spirit of the recent paper by Brightwell and Luczak \cite{brightwell2012}. There the authors analyze the degree distribution of a preferential attachment tree nearly all the way to the maximum degree. Proving something similar in our setting would require modifications in Lemma \ref{Martingale:Bound}, where the quadratic variation of the martingale for $N_n(k)$ is controlled in a wasteful manner via $V_n$. 

Another kind of question is to study the distribution of the largest part sizes in $\sP_n$. We would like to obtain such results and apply them to the ``Hollywood model" of complex networks recently proposed by Crane and Dempsey \cite{crane2017}. 
 
\appendix

\section{Some estimates on $\Gamma(x)$}
In this appendix we prove some useful bounds regarding gamma functions and other relations involving them.
\subsection{Preliminaries estimates}
\begin{lemma}[Stirling formula for Gamma function - see formula 6.1.42 in \cite{abramowitz_stegun}] \label{Stirling}
For all $x>0$ we have $$ \frac{(2\pi)^{1/2}}{e^{x}}x^{x-\frac{1}{2}} \leq \Gamma(x) \leq \frac{(2\pi)^{1/2}e^{1/12x}}{e^{x}}x^{x-1/2}.$$
\end{lemma}
\begin{lemma} \label{Approx: Stirling} For all positive $x$, it follows that
\begin{align*}
 \Gamma(x) = \frac{(2\pi)^{1/2}}{e^{x}}x^{x-\frac{1}{2}} \left( 1+ O\left(\frac{1}{x} \right) \right).
\end{align*}
\end{lemma}
\begin{proof}
Observe that by the Lemma \ref{Stirling} 
\begin{align*}
 0 \leq \Gamma(x) - \frac{(2\pi)^{1/2}}{e^{x}} x^{x-\frac{1}{2}}  \leq  \frac{(2\pi)^{1/2}}{e^{x}}x^{x-\frac{1}{2}} \left( e^{1/12x} - 1 \right),
\end{align*}
and the result follows by Taylor approximation.
\end{proof}

\begin{lemma} \label{gamma}
Let $\beta, \lambda$ be two positive real numbers with $\beta> \lambda$ then
\begin{enumerate}
\item $\displaystyle \frac{\Gamma(\beta-\lambda)}{\Gamma(\beta)} \leq e^{\frac{1}{12(\beta-\lambda)}} \left( \frac{\beta}{\beta-\lambda}\right)^{1/2} \left(\frac{1}{\beta-\lambda} \right)^{\lambda}$;

\item $\displaystyle \frac{\Gamma(\beta)}{\Gamma(\beta-\lambda)} \leq e^{\frac{1}{12\beta}} \left( \frac{\beta - \lambda}{\beta}\right)^{1/2} \beta^{\lambda}$.
\end{enumerate}
\end{lemma}

\begin{proof} For the first item, by Lemma \ref{Stirling} and the bound $(1-\frac{x}{n})^n \leq e^{-x}$ it follows that
\begin{align*}
\frac{\Gamma(\beta-\lambda)}{\Gamma(\beta)} &\leq \frac{e^{\frac{1}{12(\beta-\lambda)}}(\beta-\lambda)^{\beta-\lambda-1/2}}{e^{\beta-\lambda}} \frac{e^{\beta}}{\beta^{\beta-1/2}}\\
&\leq \frac{e^{\frac{1}{12(\beta-\lambda)}}}{e^{-\lambda}}\left( 1-\frac{\lambda}{\beta}\right)^{\beta} \left( 1+\frac{\lambda}{\beta-\lambda}\right)^{1/2} (\beta-\lambda)^{-\lambda}\\
&\leq e^{\frac{1}{12(\beta-\lambda)}} \left( \frac{\beta}{\beta-\lambda}\right)^{1/2} \left(\frac{1}{\beta-\lambda} \right)^{\lambda}.
\end{align*}
The second item follows analogously.
%\begin{align*}
%\frac{\Gamma(\beta)}{\Gamma(\beta-\lambda)} &\leq e^{\frac{1}{12\beta}} \frac{\beta^{\beta-1/2}}{e^{\beta}} \frac{e^{\beta-\lambda}} {(\beta-\lambda)^{\beta-\lambda-1/2}}\\
% &\leq e^{\frac{1}{12\beta}}e^{-\lambda} \left( 1+\frac{\lambda}{\beta-\lambda}\right)^{\beta-\lambda}\left( 1-\frac{\lambda}{\beta}\right)^{1/2} \beta^{\lambda}\\
% &\leq e^{\frac{1}{12\beta}} \left( \frac{\beta - \lambda}{\beta}\right)^{1/2} \beta^{\lambda}.
%\end{align*}
\end{proof}

\begin{lemma} \label{BinomExpon}
For $0<x<1$ and $y>0$ we have
\begin{align*}
(1-x)^y &= e^{-xy}(1+O(y^2x^3)).
\end{align*}
\end{lemma}

\begin{proof}

Observe that $(1-x)^y  = \exp(y \cdot \log(1-x) )$. Recalling the Taylor expansion of $\log$ 
\begin{align*}
\log(1-x) = -x -O(x^2)
\end{align*}
we have
\begin{align*}
(1-x)^y  &= \exp(-xy -O(yx^2) ) = (1-O(yx^2))\exp(-xy ).  
\end{align*}

\end{proof}

\begin{lemma} \label{gammagamma}
For $k=O(n^{\frac{\alpha}{2\alpha+4}})$ we have
\begin{align*}
\frac{\Gamma(n+\theta-k+\alpha)}{\Gamma(n+\theta)} &=  \frac{1}{n^{k-\alpha}}  \left(1 + O\left(\frac{k^2}{n-k} \right) \right).
\end{align*}
\end{lemma}

\begin{proof}
Using the expression given by Lemma \ref{Approx: Stirling}, we obtain
%\begin{align*}
%\Gamma(n+\theta-k+\alpha) = \frac{(2\pi)^{1/2} (n+\theta-k+\alpha)^{n+\theta-k+\alpha - \frac{1}{2} } }{e^{n+\theta-k+\alpha}} \left( 1+ O\left(\frac{1}{n+\theta-k+\alpha} \right) \right).
%\end{align*}
%
%\begin{align*} 
%\Gamma(n+\theta) = \frac{(2\pi)^{1/2} (n+\theta)^{n+\theta - \frac{1}{2}} }{e^{n+\theta}} \left( 1+ O\left(\frac{1}{n+\theta} \right) \right).
%\end{align*}
\begin{align*}
\frac{\Gamma(n+\theta-k+\alpha)}{\Gamma(n+\theta)} =  \frac{e^{k-\alpha}}{(n+\theta-k+\alpha)^{k-\alpha}} \left(1-\frac{k-\alpha}{n+\theta} \right)^{n+\theta-\frac{1}{2}} \frac{1+ O(\frac{1}{n+\theta-k+\alpha})}{1+ O(\frac{1}{n+\theta})}.
\end{align*}
Now, multiplying and dividing by $n^{k-\alpha}$ the right-hand side of the above identity becomes
\begin{align*}
\frac{e^{k-\alpha}}{n^{k-\alpha}} \left(1+\frac{k- \theta-\alpha}{n+\theta-k+\alpha} \right)^{k-\alpha} \left(1-\frac{k-\alpha}{n+\theta} \right)^{n+\theta-\frac{1}{2}} \frac{1+ O(\frac{1}{n+\theta-k+\alpha})}{1+ O(\frac{1}{n+\theta})}.
\end{align*}
Moreover, by the Lemma \ref{BinomExpon} it follows
\begin{align*}
 1 \leq \left(1+\frac{k- \theta-\alpha}{n+\theta-k+\alpha} \right)^{k-\alpha}  \leq \exp \left( \frac{(k-\alpha)(k- \theta-\alpha)}{n+\theta-k+\alpha} \right)
= 1 + O\left(\frac{k^2}{n-k} \right).
\end{align*}
Also, by Lemma \ref{BinomExpon}, for $x= \frac{k-\alpha}{n+\theta}$ and $y=n+\theta$ we have
\begin{align*}
\left(1-\frac{k-\alpha}{n+\theta} \right)^{n+\theta} &=  \exp \left( -(n+\theta)  \frac{k-\alpha}{n+\theta} \right) \exp \left( O\left( \frac{(k-\alpha)^2}{n+\theta} \right) \right)\\
&=  e^{-k+\alpha} \left( 1+  O\left( \frac{k^2}{n} \right) \right),
\end{align*}
and for $k=O(n^{\frac{\alpha}{2\alpha+4}})$
\begin{align*}
 e^{-k+\alpha} & \left( 1+  O\left( \frac{k^2}{n} \right) \right) \left(1-\frac{k-\alpha}{n+\theta} \right)^{-\frac{1}{2}} \left(1 + O\left(\frac{k^2}{n-k} \right) \right) \frac{1+ O(\frac{1}{n+\theta-k+\alpha})}{1+ O(\frac{1}{n+\theta})} \\
&=  e^{-k+\alpha}  \left(1 + O\left(\frac{k^2}{n} \right) \right).
\end{align*}
Thus
\begin{align*}
\frac{\Gamma(n+\theta-k+\alpha)}{\Gamma(n+\theta)} &=  \frac{e^{k-\alpha}}{n^{k-\alpha}}e^{-k+\alpha} \left( 1+  O\left( \frac{k^2}{n} \right) \right) =  \frac{1}{n^{k-\alpha}}  \left(1 + O\left(\frac{k^2}{n} \right) \right).
\end{align*}
\end{proof}

\subsection{Order of $\phi_n$ and $\psi_n(k)$}
This part is devoted to prove bounds for the two normalizing factors $\phi_n$ and $\psi_n(k)$ whose definition we recall latter.
\[
\phi_n = \frac{\Gamma(1+\theta)}{\Gamma(1+\theta+\alpha)}\,\frac{\Gamma(n+\alpha+\theta)}{\Gamma(n+\theta)},
\]

\begin{lemma} \label{Ord:phin}
Let $\phi_n$ be as above, then the following bounds hold
\begin{enumerate}
\item $\displaystyle \frac{1}{\phi_j} < \frac{2\Gamma(1+\theta+\alpha)}{\Gamma(1+\theta) \cdot (j+\theta)^{\alpha}}$;

\item $\displaystyle \frac{1}{(j+\theta) \phi_{j+1}} < \frac{2\Gamma(1+\theta+\alpha)}{\Gamma(1+\theta) \cdot (j+\theta)^{1+\alpha}}$;

\item There exists a constant $C_{\phi}$ such that
$$ \phi_j \leq C_{\phi} j^{\alpha};$$
\end{enumerate}
In particular $\phi_n = \Theta(n^{\alpha})$.
\end{lemma}

\begin{proof}
Let us prove the first two items and the third will follow analogously.

\noindent \textit{(1).} By Lemma \ref{gamma}
\begin{align*}
\frac{\Gamma(j+\theta)}{\Gamma(j+\theta+\alpha)} &\leq e^{\frac{1}{12(j+\theta)}}\left(1+\frac{\alpha}{j+\alpha+\theta} \right)^{1/2}  (j+\theta)^{-\alpha} \leq 2(j+\theta)^{-\alpha}.
\end{align*}
then
\begin{align*}
\frac{1}{\phi_j} < \frac{2\Gamma(1+\theta+\alpha)}{\Gamma(1+\theta) \cdot (j+\theta)^{1+\alpha}}.
\end{align*}

\noindent \textit{(2).} This part follows using the duplication property $\Gamma(x+1)=x \Gamma(x)$ and the previous item and the inequality
\begin{align*}
\frac{\Gamma(j+1+\theta)}{\Gamma(j+1+\theta+\alpha)} = \frac{(j+\theta) \Gamma(j+\theta)}{(j+\theta+\alpha) \Gamma(j+\theta+\alpha)} < \frac{\Gamma(j+\theta)}{\Gamma(j+\theta+\alpha)}.
\end{align*}

\end{proof}
The next lemma provides similar bounds for the normalization factor $\psi_n(k)$ whose definition is recalled bellow.
\begin{equation*}
\begin{split}
\psi_n(k) =  \frac{\Gamma(k+\theta)\Gamma(n-k+\alpha+\theta)}{\Gamma(\alpha+\theta)\Gamma(n+\theta)}.
\end{split}
\end{equation*}
\begin{lemma} \label{Bound:psik}
For $\psi_n(k)$ defined as above, the following bounds hold
\begin{enumerate} 
\item $\psi_n(k) \leq \dfrac{2\Gamma(k+\theta)}{ \Gamma(\alpha+\theta)} \dfrac{1}{(n+\theta-k+\alpha)^{k-\alpha}}$, for $n \geq 2k$;

\item $\dfrac{1}{\psi_n(k)} \leq \dfrac{ e^{\frac{1}{12}} \Gamma(\alpha+\theta)}{\Gamma(k+\theta)} (n +\theta)^{k-\alpha}$

\end{enumerate}
\end{lemma}

\begin{proof}
\noindent \textit{(1).} By Lemma \ref{gamma} we have
\begin{align*}
\frac{\Gamma(n-k+\alpha+\theta)}{\Gamma(n+\theta)} &\leq e^{\frac{1}{12(n+\theta-k+\alpha)}}  \left( 1+ \frac{k-\alpha}{n+\theta  - k + \alpha}\right)^{1/2} \frac{1}{(n+\theta - k +\alpha)^{k-\alpha}}.
\end{align*}
And for $2k \leq n$ it follows that
\begin{align*}
\frac{\Gamma(n-k+\alpha+\theta)}{\Gamma(n+\theta)} &\leq  \frac{2}{(n+\theta - k +\alpha)^{k-\alpha}}.
\end{align*}
Then 
\begin{equation*}
\psi_n(k) \leq \dfrac{2\Gamma(k+\theta)}{ \Gamma(\alpha+\theta)} \dfrac{1}{(n+\theta-k+\alpha)^{k-\alpha}}.
\end{equation*}

2. Again, by Lemma \ref{gamma} we have
\begin{align*}
\frac{\Gamma(n+\theta)}{\Gamma(n-k+\alpha+\theta)} &\leq e^{\frac{1}{12(n+\theta)}}  \left( 1- \frac{k-\alpha}{n+\theta}\right)^{1/2} (n+\theta)^{k-\alpha} \leq e^{\frac{1}{12}} (n+\theta)^{k-\alpha}
\end{align*}
and the result follows from the previous inequality.
\end{proof}

\begin{lemma} \label{phipsi} For the ration of the factors $\phi_n$ and $\psi_n(k)$ the following upper bound holds
\begin{align*}
\frac{\phi_j}{(\psi_{j+1}(k))^2 \cdot(j+\theta)} \leq \frac{\Gamma(1+\theta)\Gamma(\alpha+\theta)^2}{\Gamma(1+\theta+\alpha)\Gamma(k+\theta)^2} (j+\theta)^{2k-\alpha-1}.
\end{align*}
\end{lemma}

\begin{proof} Using the definition of both factors, we have
\begin{align*}
\frac{\phi_j}{(\psi_{j+1}(k))^2 (j+\theta)} = \frac{\Gamma(1+\theta)\Gamma(\alpha+\theta)^2}{\Gamma(1+\theta+\alpha)\Gamma(k+\theta)^2}  \frac{\Gamma(j+\theta)\Gamma(j-1+\alpha+\theta)}{\Gamma(j+1-k+\alpha+\theta)^2} (j+\theta)^2(j+\alpha+\theta)
\end{align*}
and using the bounds on ratio of gamma functions in Lemma \ref{gamma}, we have
\begin{align*}
 \frac{\phi_j}{(\psi_{j+1}(k))^2 \cdot(j+\theta)} &\leq \frac{e^{\frac{1}{12}}\Gamma(1+\theta)\Gamma(\alpha+\theta)^2}{\Gamma(1+\theta+\alpha)\Gamma(k+\theta)^2} (j+\theta)^{2k-\alpha-1}.
\end{align*}
\end{proof}
%\newpage
%\appendix

\bibliographystyle{plain}
%\nocite{*}
\bibliography{refs}

\end{document}